\documentclass[11pt]{article}
\usepackage{amsmath, amssymb, amsthm, amsfonts, amscd}
\usepackage{indentfirst}
\usepackage[mathscr]{euscript}
\usepackage{xcolor}
\usepackage[pdftex]{graphicx}
\usepackage{float}
\usepackage{enumitem}

\numberwithin{equation}{section}
\numberwithin{figure}{section}

\theoremstyle{plain}
\newtheorem{theorem}{Theorem}\numberwithin{theorem}{section}
\newtheorem{lemma}{Lemma}\numberwithin{lemma}{section}
\newtheorem{proposition}{Proposition}\numberwithin{proposition}{section}
\newtheorem{corollary}{Corollary}\numberwithin{corollary}{section}
\theoremstyle{definition}
\numberwithin{definition}{section}
\theoremstyle{remark}
\newtheorem{remark}{Remark}\numberwithin{remark}{section}

\newcommand{\R}{\mathbb{R}}
\newcommand{\hj}{{\mathbf{j}}}
\newcommand{\F}{{}_1F_2}

\newcommand{\rb}[1]{\left(#1\right)}

\newcommand{\qb}[1]{\left[#1\right]}
\newcommand{\vb}[1]{\left\vert#1\right\vert}

\newcommand{\bJ}{\mathbb{J}}
\newcommand{\bH}{\mathbb{H}}

\title{Partial fraction expansions\\ and zeros of Hankel transforms}

\author{Yong-Kum Cho\footnote{ykcho@cau.ac.kr. Department of Mathematics, College of Natural Sciences, Chung-Ang University,
84 Heukseok-Ro, Dongjak-Gu, Seoul 06974, Korea.}
\and  Seok-Young Chung\footnote{seok-young.chung@ucf.edu. Department of Mathematics, University of Central Florida,
4393 Andromeda Loop N., Orlando, FL 32816, USA.}\\
\and Young Woong Park\footnote{ywpark1839@gmail.com. Department of Mathematics, College of Natural Sciences, Chung-Ang University,
84 Heukseok-Ro, Dongjak-Gu, Seoul 06974, Korea.}}

\date{}

\begin{document}

\maketitle

{\bf Abstract.} It is proved by the method of partial fraction expansion
and Sturm's oscillation theory that the zeros of certain Hankel transforms are
all real, simple and distributed one by one between consecutive zeros of Bessel functions.
As an application, we obtain a list of sufficient conditions as well as necessary conditions on parameters
for which ${}_1F_2$ hypergeometric functions belong to the Laguerre-P\'olya class.

\medskip


{Laguerre-P\'olya class, partial fraction expansions, transference principle.}

\medskip


\section{Introduction}
This paper deals with partial fraction expansions for the ratios of Hankel transforms to Bessel functions
and its applications to the theory of zeros of Hankel transforms, which extend those analogous results of
Hurwitz and P\'olya \cite{P} concerning Fourier cosine and sine transforms.

The Hankel transform under consideration is defined by
\begin{equation}\label{H1}
\mathcal{H}_\nu(f)(z) = \int_0^1 f(t)J_\nu(zt)\sqrt{zt}\, dt
\end{equation}
for $\,z\in\mathbb{C}\setminus\{0\},$ where $J_\nu(z)$ stands for the Bessel function of the first kind of order $\nu$
and $f(t)$ is a real-valued function supported in the unit interval.
It will be assumed throughout that $\nu$ is real with $\,\nu>-1.$

From a complex analysis point of view, it is advantageous to consider
the \emph{normalized Hankel transform} defined by
\begin{equation}\label{H2}
\bH_\nu (f)(z) = \int_0^1 t^{\nu+1/2} f(t) \bJ_\nu(zt) \,dt\quad(z\in\mathbb{C}),
\end{equation}
where the kernel $\bJ_\nu(z)$ stands for the entire function given by
\begin{equation*}
\bJ_\nu(z) = \sum_{m=0}^\infty \frac{(-1)^m (z/2)^{2m}}{m!\,(\nu+1)_m} = \Gamma(\nu+1)(z/2)^{-\nu} J_\nu(z).
\end{equation*}

 Due to the readily-verified relation
\begin{equation}\label{H3}
\,\mathcal{H}_\nu(f)(z) = \frac{\sqrt 2\,(z/2)^{\nu+1/2}}{\Gamma(\nu+1)}\,\bH_\nu(f)(z),
\end{equation}
it is evident that $\,J_\nu(z), \bJ_\nu(z)$ share zeros in common for $\,z\ne 0\,$
and so do Hankel transforms $\,\mathcal{H}_\nu(f)(z), \bH_\nu(f)(z).$ A distinctive aspect is that the normalized Hankel transform
$\bH_\nu(f)(z)$, if exists, is a real entire function.

In the special cases $\,\nu =\pm 1/2,$  \eqref{H1} reduces to
\begin{equation}\label{HP1}
U(z) = \int_0^1 f(t)\cos zt\,dt,\quad V(z) = \int_0^1 f(t)\sin zt\,dt,
\end{equation}
except multiplicative factor $\sqrt{2/\pi\,},$ the Fourier cosine and sine transforms, respectively.\footnote{
For this reason, the Hankel transform \eqref{H1} is often referred to as the Fourier-Bessel or generalized Fourier transform.}
In a simplified version, it is shown by Hurwitz and P\'olya \cite{P} that $U(z), V(z)$ have an infinity of zeros which are all real and simple
when $f(t)$ is positive and increasing for $\,0<t<1,$ unless it is a step function.
In addition, $U(z), V(z)$ have exactly one zero in each of intervals
$$\big((m-1/2)\pi,\,(m+1/2)\pi\big),\quad \big(m\pi,\,(m+1)\pi\big),\quad m=1, 2, \cdots,$$
separately, and have no zeros elsewhere.

The idea of proof is based on partial fraction expansions
\begin{align}\label{HP2}
\frac{U(z)}{\,z\cos z\,} &= \frac{U(0)}{z} + \sum_{m=1}^\infty\frac{(-1)^m U\left[(m-1/2)\pi\right]}{(m-1/2)\pi}\nonumber\\
&\qquad\qquad\times\,\left[\frac{1}{z-(m-1/2)\pi} + \frac{1}{z+ (m-1/2)\pi}\right],
\end{align}
\begin{equation}\label{HP3}
\frac{V(z)}{\,z\sin z\,} = \frac{V'(0)}{z} +
\sum_{m=1}^\infty (-1)^m V(m\pi)\left(\frac{1}{z-m\pi} + \frac{1}{z+ m\pi}\right),
\end{equation}
both of which are valid if $f(t)$ is integrable. By observing that the Fourier coefficients
$\,U\left[(m-1/2)\pi\right], \,V(m\pi)\,$ alternate in sign when $f(t)$ is positive and increasing,
Hurwitz and P\'olya obtained the stated results by inspecting partial sums of \eqref{HP2}, \eqref{HP3}
(see section 3 for the detail).

In this paper we aim at investigating the nature and distribution of zeros of Hankel transforms in an analogous manner.
For this purpose, we shall establish partial fraction expansions of the form
\begin{equation*}\label{H5}
 \frac{\,\mathcal{H}_\nu (f)(z)\,}{z^\lambda J_\mu (z)}  =\frac{b}{z} +
 \sum_{m=1}^\infty c_m\rb{ \frac{1}{z-j_{\mu,m}} + \frac{1}{z+j_{\mu,m}}}
 \end{equation*}
 in the range $\,-1<\mu<\nu+2,$ where $\big(j_{\mu, m}\big)$ denotes all positive zero of $J_\mu(z)$ and
 $\lambda$ is chosen so as to determine
$\,b, c_m\,$ explicitly. By exploiting these expansions, we shall modify the method of Hurwitz and P\'olya
to prove that the zeros of $\mathcal{H}_\nu(f)(z)$ are distributed one by one between consecutive zeros of $J_\mu(z)$,
provided that $c_m$ keeps constant sign.

For general $\mu, \nu$ in the range $\,-1<\mu<\nu+2,$
it is seemingly impossible to determine whether $c_m$ keeps constant sign.
In the particular case $\,\mu =\nu,$ however, we shall give a set of
sufficient conditions on $f(t)$ deduced on the basis
of Sturm's comparison theorems and its modification.

A great deal of special functions encountered in mathematical physics arise as the images of
Hankel transforms of type \eqref{H1} or \eqref{H2}. In this paper we shall focus on
$\F$  hypergeometric functions of the form
$$\Phi(z)= {}_1F_2 \left[\begin{array}{c}
a\\ b, \,c\end{array}\biggr| - \frac{\,z^2}{4} \right ] \quad(z\in\mathbb{C}),$$
which can be identified as the images of \eqref{H2} under certain circumstances.
In application of our results, we shall focus on determining the region of
$\,(a, b, c)\in\R_+^3\,$ for which $\Phi(z)$ belongs to the Laguerre-P\'olya class, that is, the class of all real entire functions $g(z)$
having only real zeros, if any, which are representable in the product form
\begin{equation}\label{LP1}
g(z) = A z^\ell e^{- \alpha z^2 + \beta z} \prod_{m=1}^\omega \left(1-\frac{z}{\sigma_m}
\right) e^{\frac{z}{\,\sigma_m\,}},\quad 0\le\omega\le \infty,
\end{equation}
where $\, A, \alpha, \beta\,$ are real with $\,\alpha\ge 0,$ $\ell$ is a nonnegative integer and $\big(\sigma_k\big)$
is a sequence of non-zero reals with $\,\sum_{m=1}^\omega 1/\sigma_m^2 <\infty\,$ (see \cite{DR}, \cite{So}).

The problem of identifying the exact  set of positive parameters for which $\Phi(z)$ belongs to the Laguerre-P\'olya class
is a long-standing open problem in the theory of entire functions (see \cite{CCs}, \cite{Hi}, \cite{So}).

To describe briefly, we shall assign to each fixed $\,a>0\,$ an unbounded hyperbolic region $\mathcal{P}_a\subset\R_+^2\,$
having the property that $\Phi(z)$ has only complex zeros for each $\,(b, c)\in\mathcal{P}_a.$  In the complement of $\mathcal{P}_a,$
we apply our results on Hankel transforms to specify the range of $(b, c)$
for which $\Phi(z)$ belongs to the Laguerre-P\'olya class (see Figure \ref{Fig2} in a particular case). Although there
still remains a large part of the complement of $\mathcal{P}_a$ left undetermined, our result provides a considerable improvement
of the known parameter patterns for the Laguerre-P\'olya class available in the literature.

We organize the present paper as follows.

In section 2, we establish the aforementioned partial fraction expansions
based on Cauchy's residue theorem. In section 3, we revisit
the method of Hurwitz and P\'olya. In sections 4, 5,
we describe the nature and distribution of zeros of Hankel transforms and identify $\bH_\nu(f)(z)$ as
a member of the Laguerre-P\'olya class under the assumption that $\bH_\nu(f)(j_{\mu, m})$ alternates in sign.
In section 6, we modify Sturm's comparison theorems to give sufficient conditions on $f(t)$ for the sign alternation of
$\bH_\nu(f)(j_{\nu, m}).$ In the last section 7, we consider $\F$ hypergeometric functions of the above type
and deal with the problem of identifying parameters for the Laguerre-P\'olya class.

\section{Partial fraction expansions}
It is classical that Bessel function $J_\nu(z)$ of order $\,\nu>-1\,$
has an infinity of zeros which are all real and simple.
In addition, the zeros are symmetrically distributed about the origin and it is standard
to denote all positive zeros by $\rb{j_{\nu, m}}$ arranged in ascending order of magnitude (\cite[\S 15.2-15.22]{Wa}).

Our principal result states as follows.

\begin{theorem}\label{theoremP1}
Let $\,\nu>-1, \,-1<\mu<\nu+2\,$ and $f(t)$ be an integrable function, defined for $\,0<t<1,$ subject to the condition
\begin{equation}\label{IA}
\int_0^1 t^{\nu+1/2}|f(t)| dt<\infty\quad\text{when}\quad \nu<- 1/2.
\end{equation}
Put $\,\lambda =\nu-\mu+3/2\,$ and $\,\mathcal{D}_\mu \equiv \mathbb{C} \setminus \left\{0, \,\pm j_{\mu,1},\,\pm j_{\mu,2},\,\cdots\right\}$.
If $\mathcal{H}_\nu(f)(z)$ denotes the Hankel transform of $f(t)$ defined by \eqref{H1}, then
\begin{align}\label{Paf1}
 &\frac{\,\mathcal{H}_\nu(f)(z)\,}{\,z^{\lambda}J_\mu (z)\,} =
 \frac{b}{z} - \sum_{m=1}^\infty a_m \rb{ \frac{1}{z-j_{\mu,m}} + \frac{1}{z+j_{\mu,m}}}\quad (z\in\mathcal{D}_\mu),\\
 &\quad\text{where}\qquad \left\{\begin{aligned}
&{ \,\,b\,\,=\frac{\,\Gamma(\mu+1)}{ 2^{\nu-\mu}\Gamma(\nu+1)}\int_0^1 t^{\nu+1/2} f(t)dt,}\\
&{a_m =  \frac{\mathcal{H}_\nu (f)\rb{j_{\mu,m}} }{\,j_{\mu,m}^{\lambda} J_{\mu+1}\rb{j_{\mu,m}}},\quad m=1, 2, \cdots.}
\end{aligned}\right.\nonumber
 \end{align}
In addition, the series of \eqref{Paf1} converges absolutely for each $\,z\in \mathcal{D}_\mu\,$ and
 the convergence is uniform on each compact subset of $\mathcal{D}_\mu$.
\end{theorem}

\subsection{Inequalities of Bessel functions}
As one of the key analytical tools in what follows, we shall establish
the following inequalities for Bessel functions of real order.

\begin{lemma}\label{lemmaP1}
If $\nu$ is real with $\,\nu>-1,$ then there exists a positive constant $\,c_\nu>0,$ depending only on $\nu$, such that
for all $\,z\in\mathbb{C}\setminus\{0\},$
\begin{align*}
\left|J_\nu(z)\right| \le c_\nu\,\frac{ \exp\left(|{\rm{Im}}\, z\right|)}{\sqrt{|z|}}\cdot
\left\{\begin{aligned}
&{\quad 1} &{\big(\nu\ge - 1/2\big),}\\
&{1+ |z|^{\nu+1/2}} &{\big(\nu<-1/2\big).} \end{aligned}\right.
\end{align*}
\end{lemma}

\begin{proof}
A version of Hankel's asymptotic formula \cite[\S 7.21]{Wa} reads
\begin{align}\label{ZH}
J_\nu(z) &= \sqrt{\frac{2}{\pi z}\,}\Bigg\{\cos \Big(z- \frac{\nu\pi}{2} -\frac{\pi}{4}\Big)
\left[ 1+ O\left(z^{-2}\right)\right]\nonumber\\
&\qquad\qquad  -\sin\Big(z- \frac{\nu\pi}{2} -\frac{\pi}{4}\Big)
\left[ \frac{4\nu^2-1}{8z} + O\left( z^{-3}\right)\right]\Bigg\}
\end{align}
as $\,z\to\infty,$ provided that $\,|\arg\, z|<\pi.$ By obvious estimates
$$
\left|\cos \Big(z- \frac{\nu\pi}{2} -\frac{\pi}{4}\Big)\right|\le \exp(|{\rm{Im}}\, z|),
\,\, \left|\sin \Big(z- \frac{\nu\pi}{2} -\frac{\pi}{4}\Big)\right|\le \exp(|{\rm{Im}}\, z|),$$
the asymptotic formula \eqref{ZH} implies that we can find $\,r>1\,$ so that
\begin{equation}\label{HL1}
|J_\nu(z)| \le \frac{ \exp\left(|{\rm{Im}}\, z\right|)}{\sqrt{|z|}}\quad\text{for all $z$ with}\quad |z|>r,
\end{equation}
provided that $\,{\rm Re}\,z\ge 0.$ For $z$ with $\,|z|>r, \,{\rm Re}\,z<0,$ if we consider rotating clockwise through the angle
$\pi$ and recall the relation $\,J_\nu(z) = e^{-\nu\pi i} J_\nu(e^{\pi i} z),$ it is easy to see that the estimate \eqref{HL1}
continues to be valid. Since the function $\,z\mapsto J_\nu(z)\,\sqrt{ |z|}\,\exp\left(-|{\rm{Im}}\, z\right|)\,$ is continuous
 for $\,z\ne 0,$ it must be bounded in the annulus $\,\{z: 1\le |z|\le r\}\,$ and hence
\begin{equation}\label{HL2}
|J_\nu(z)| \le d_\nu\, \frac{ \exp\left(|{\rm{Im}}\, z\right|)}{\sqrt{|z|}}\quad\text{for all $z$ with}\quad  |z|\ge 1
\end{equation}
for some positive constant $d_\nu$ which depends only on $\nu$.

Let $\,\nu\ge -1/2.$ By applying the known inequality  \cite[\S 3.31, (1)]{Wa}
$$\left|J_\nu(z)\right| \le \frac{\,|z|^\nu \exp(|{\rm{Im}}\,z|)}{2^\nu \Gamma(\nu+1)}\,$$
and noting the fact that $\,|z|^{\nu+1/2}\le 1\,$ when $\,|z|\le 1,$ we deduce
\begin{equation}\label{HL3}
|J_\nu(z)| \le \frac{1}{\,2^\nu \Gamma(\nu+1)} \frac{ \exp\left(|{\rm{Im}}\, z\right|)}{\sqrt{|z|}}\quad\text{for all $z$ with}\quad  |z|\le 1.
\end{equation}
On combining \eqref{HL2}, \eqref{HL3}, the stated estimate follows with
$$c_\nu = \max \left(d_\nu,\, \frac{1}{\,2^\nu \Gamma(\nu+1)} \right).$$

In the case $\,-1<\nu<-1/2,$ we use \cite[\S 3.31, (2)]{Wa} to estimate
\begin{align}\label{HL4}
\left|J_\nu(z)\right| &\le \frac{\,|z|^\nu\exp(|{\rm{Im}}\,z|)}{2^\nu \Gamma(\nu+1)}
\left[1+ \frac{|z|^2}{4(\nu+1)(\nu+2)}\right]\nonumber\\
&\le \frac{\,(2\nu+3)^2\,}{2^{\nu+2}\Gamma(\nu+3)}\frac{\,|z|^{\nu+1/2}\exp(|{\rm{Im}}\,z|)}{\sqrt{|z|}}
\end{align}
for $\,|z|\le 1\,$ and the desired estimate follows on combining \eqref{HL2}, \eqref{HL4} with
$$c_\nu = \max \left(d_\nu,\,  \frac{\,(2\nu+3)^2\,}{2^{\nu+2}\Gamma(\nu+3)}\right).$$
\end{proof}

In dealing with the convergence matter, we shall need uniform lower bounds
for the values of derivatives of Bessel functions at positive zeros.

\begin{lemma}\label{lemmaP2}
For real $\,\mu>-1,$ the following estimates hold true.
\begin{align*}
\left|\sqrt{j_{\mu, m}} \,J_{\mu+1}\left(j_{\mu, m}\right)\right| \ge \left\{\begin{aligned}
&{\sqrt{\frac{2}{\pi}\,}\quad\text{if $\,|\mu|\le 1/2,\,\, m = 1, 2, \cdots,$}}\\
&{\frac{1}{\mu\sqrt{2\pi}}\quad\text{if $\,\mu>1/2,\,\,m = 1, 2, \cdots,$}}\\
&{\sqrt{\frac{15}{8\pi}\,}\quad\text{if $\,\mu<-1/2,\,\,m = 2, 3, \cdots.$}}
\end{aligned}\right.
\end{align*}
\end{lemma}

\begin{proof}
On making use of the known expressions \cite[\S 3.4]{Wa}
\begin{align}\label{SJ}
\begin{split}
J_{-1/2}(z) &= \sqrt{\frac{2}{\pi z}\,} \cos z,\quad J_{1/2}(z) = \sqrt{\frac{2}{\pi z}\,} \sin z,\\
J_{3/2}(z) &=  \sqrt{\frac{2}{\pi z}\,}\left(\frac{\sin z}{z} -\cos z\right),
\end{split}
\end{align}
the estimates in the cases $\,|\mu| =1/2\,$ are trivially verified and hence it suffices to
prove the estimates in the remaining cases.

The function $\,u(x) \equiv \sqrt{x}\,J_\mu(x)\,$ solves the differential equation
\begin{equation}\label{SL1}
u'' + \phi(x) u =0,\quad x>0,\quad\text{where $\,\,
\phi(x) = 1+\frac{1/4-\mu^2}{x^2}\,.$}
\end{equation}
We note that $\phi(x)$ decreases to the value $1$ when $\,|\mu|<1/2\,$ and increases
to the value $1$ when $\,|\mu|>1/2\,$ with a unique zero at $\,x=\sqrt{\mu^2-1/4\,}.$

\smallskip

(i) Assuming $\,|\mu|<1/2,$ we consider the auxiliary function
$$g(x) = \phi(x)[u(x)]^2 + [u'(x)]^2,\quad x>0$$
for which $\,g'(x) = \phi'(x)[u(x)]^2\,$
due to the differential equation \eqref{SL1}.

According to Hankel's asymptotic formula \cite[\S 7.21]{Wa}, we have
\begin{align*}
u(x) &= \sqrt{\frac{2}{\pi}}\,\cos \chi_\mu + O\left(x^{-1}\right),\\
u'(x) &= \frac{1}{2\sqrt x} J_\mu(x) + \sqrt x\,J_\mu'(x)\\
&=  -\sqrt{\frac{2}{\pi}}\,\sin\chi_\mu + O\left(x^{-1}\right)
\end{align*}
as $\,x\to\infty,$ where $\,\chi_\mu = x-(\mu/2 +1/4)\pi.$ On combining with the trivial behavior
$\,\phi(x) = 1 + O\left(x^{-2}\right),$ we find that
$$ g(x) = \frac{2}{\pi} + O\left(x^{-1}\right)\quad\text{as $\,\, x\to\infty.$}$$
It is thus found that $g(x)$ decreases to the value $\,2/\pi\,$ and hence we deduce that
$\,g(j_{\mu, 1})>g(j_{\mu, 2})>\cdots \,\to\, 2/\pi.$ Since
$$g(j_{\mu, m}) = \big[u'(j_{\mu, m})\big]^2 = j_{\mu, m} \big[J_{\mu+1}(j_{\mu, m})\big]^2,$$
we obtain the uniform lower bound
\begin{equation*}
\sqrt{ j_{\mu, m}}\,\big|J_{\mu+1}(j_{\mu, m})\big| \ge \sqrt{\frac{2}{\pi}\,},\quad m=1, 2, \cdots.
 \end{equation*}

(ii) In the case $\,|\mu|>1/2,$ we consider the auxiliary function
$$h(x) = [u(x)]^2 + \frac{1}{\phi(x)} [u'(x)]^2,\quad x>\sqrt{\mu^2 - 1/4\,}.$$
As readily calculated with the aid of \eqref{SL1}, we have
$$h'(x) = -\frac{\phi'(x)}{[\phi(x)]^2} [u'(x)]^2, \,\,\frac{1}{\phi(x)} = 1 + O\left(x^{-2}\right)\quad\text{as $\,x\to\infty,$}$$
and thus, by the same reasoning as above, $h(x)$ decreases to the value $2/\pi$.
Evaluating at $\,x= j_{\mu, m},$ we deduce the lower bound
$$\sqrt{ j_{\mu, m}}\,\big|J_{\mu+1}(j_{\mu, m})\big| \ge \sqrt{\frac{2}{\pi}\,}\sqrt{\phi(j_{\mu, m})\,,}$$
provided that $\,j_{\mu, m} >\sqrt{\mu^2 -1/4}.$
\begin{itemize}
\item[(a)]
In view of Lorch's estimate for the first positive zero (\cite{Lo})
$$ j_{\mu, 1}^2 > (\mu+1)(\mu+5),\quad \mu>-1,$$
we observe that $\,j_{\mu, 1} >\mu\,$ when$\,\mu>1/2.$ Since $\phi(x)$ is increasing and
$\,\phi(x)>1/(4\mu^2)\,$ for all $\,x>\mu,$ we find that if $\,\mu>1/2,$ then
\begin{equation*}
\sqrt{ j_{\mu, m}}\,\big|J_{\mu+1}(j_{\mu, m})\big| \ge \frac{1}{\mu \sqrt{2\pi}}, \quad m=1, 2, \cdots.
\end{equation*}
\item[(b)]
In the case $\,\mu<-1/2,$ while it is known that $\,j_{\mu, 1} \to 0\,$ as $\,\mu\to -1,$ we have
$\,j_{\mu, 2} \ge j_{1, 1} \,$ (see \cite{CC1}). Since $\,j_{1, 1}>\sqrt{12}\,$ according to the above Lorch's estimate
and $\,\phi(\sqrt{12}) > 15/16,$ we find that
\begin{equation*}
\sqrt{ j_{\mu, m}}\,\big|J_{\mu+1}(j_{\mu, m})\big| \ge \sqrt{\frac{15}{8\pi}}, \quad m=2, 3, \cdots.
\end{equation*}
\end{itemize}

On collecting the above case estimates, we complete the proof.
\end{proof}

\subsection{Proof of Theorem \ref{theoremP1}}
Our proof of Theorem \ref{theoremP1}
will be based on Cauchy's residue theorem and preceding inequalities of Bessel functions.
For the sake of convenience, we shall divide our proof into four different stages.

\medskip
{\bf 1.} We fix a point $\,z\in \mathcal{D}_\mu\,$ and define
\begin{align*}
\psi(w) &= \frac{\mathcal{H}_\nu(f)\rb{w}}{\,w^{\lambda}\rb{w-z}J_\mu (w)\,} = \int_0^1 f(t)\Psi(w, t)\, dt,\\
&\text{where}\quad \Psi(w, t) =\frac{J_\nu(wt)\sqrt{wt}}{\,w^\lambda \rb{w-z} J_\mu(w)\,}.
\end{align*}

In view of the limiting behavior
\begin{equation}\label{Pf1}
\Psi(w, t) \,\sim\, -\frac{\Gamma(\mu+1)}{\,2^{\nu-\mu}\,\Gamma(\nu+1)\,z}
\frac{t^{\nu+1/2}}{w}\quad\text{as}\quad w\to 0
\end{equation}
and the fact that the zeros of $J_\mu(w)$ are all simple, it is
evident that the function $\psi(w)$ is meromorphic with simple poles at
$\,z,\,0,\, \pm j_{\mu,1},\, \pm j_{\mu,2},\,\cdots.\,$
By making use of \eqref{Pf1}, it is trivial to calculate the residues
\begin{align*}
 \mathrm{Res}\qb{\psi\rb{w};w=z}  = \frac{\mathcal{H}_\nu(f)\rb{z}}{z^{\lambda}J_\mu (z)},\quad
 \mathrm{Res}\qb{\psi\rb{w};w=0} = -\frac{b}{z},
 \end{align*}
where it is assumed with no loss of generality that $\,b\ne 0.$
By using the relation
$\,J_\mu'\left(j_{\mu, m}\right) = - J_{\mu+1}\left(j_{\mu, m}\right),$
it is also easy to calculate
\begin{align*}
\mathrm{Res}\qb{\psi\rb{w};w= j_{\mu,m}} = \frac{a_m}{z - j_{\mu, m}}.
\end{align*}

We note that if $j$ denotes any positive zero of
$J_\mu(w)$, then
\begin{align}\label{Pf2}
\lim_{w\to -j} (w+j)\Psi(w, t) &= \frac{\,J_\nu(-jt)\sqrt{-jt}}{(-j)^\lambda (j+z) J_{\mu+1}(-j)\,}\nonumber\\
&=\frac{\, e^{(\nu -\lambda -\mu -1/2)\pi i}\,J_\nu(jt)\sqrt{jt}\,}{j^\lambda (j+z) J_{\mu+1}(j)}\nonumber\\
& = \frac{\,J_\nu(jt)\sqrt{jt}\,}{j^\lambda (j+z) J_{\mu+1}(j)},
\end{align}
where we have used the relation $\,J_\mu(-w) = e^{\mu\pi i} J_\mu(w)\,$ (\cite[\S 3.62]{Wa}), whence
\begin{align*}
\mathrm{Res}\qb{\psi\rb{w};w=- j_{\mu,m}} = \frac{a_m}{z +j_{\mu, m}}.
\end{align*}

For each positive integer $n$, let $R_n$ be the rectangle
with vertices at $\,\pm X_n \pm Yi,$ where $\,X_n = \left(n+ \mu/2+ 1/4\right)\pi\,$
and $\,Y>2|z|.$ Due to McMahon's asymptotic formula \cite{Mc} which states
\begin{equation}\label{ZA}
j_{\mu, n} = \left(n+ \frac \mu 2 - \frac 14\right)\pi + O\big(n^{-1}\big)\quad \text{as}\quad n\to \infty,
\end{equation}
there exists an integer $n_0$ such that $\,2|z|<j_{\mu, n}<X_n <j_{\mu, n+1}\,$ when $\,n\ge n_0.$ By the residue theorem,
hence, if $\,n\ge n_0,$ then
\begin{align}\label{Pf3}
&\frac{1}{2\pi i}\int_{R_n} \psi(w) dw\nonumber\\
&\quad =  \frac{\mathcal{H}_\nu(f)(z)}{z^{\lambda}J_\mu (z)} - \frac{b}{z}
+\sum_{m=1}^n a_m \rb{ \frac{1}{z-j_{\mu,m}} + \frac{1}{z+j_{\mu,m}}}.
\end{align}

\medskip

{\bf 2.} We shall now prove that the integrals of $\psi(w)$ along the upper and lower sides of $R_n$
tend to zero as $\,Y\to \infty\,$ for each fixed $\,n\ge n_0$.

\begin{lemma}\label{lemmaP3}
For real $\,\nu>-1, \,\mu>-1,$ there exists a positive constant $A$, depending only on $\mu, \nu$, such that
\begin{align}\label{BE3}
\left|\frac{\,J_\nu(wt) \sqrt{wt}\,}{J_\mu(w)}\right|  \le A\,|w|^{1/2}\cdot\left\{\begin{aligned}
&{\quad 1\qquad \,\,\,\big(\nu\ge - 1/2\big),}\\
&{t^{\nu+1/2}\,\,\quad\,\big(\nu<-1/2\big)} \end{aligned}\right.
\end{align}
for all $w$ with $\,\left|{\rm{Im}}\, w\right|\ge 1\,$ and for all $\,0<t<1.$
\end{lemma}

\begin{proof}
By making use of the inequality
$$ \left|\cos \Big(z- \frac{\mu\pi}{2} -\frac{\pi}{4}\Big)\right| \ge \frac 14 \exp(|{\rm Im\,} w|),$$
valid when $\,|{\rm Im\,} w|\ge 1,$ if we apply Hankel's asymptotic formula \eqref{ZH} in the same manner as before,
 then we can find $\,r>1\,$ such that
\begin{equation*}
\left|J_\mu(w)\right| \ge \frac 16  \frac{\exp(|{\rm Im\,} w|)}{\sqrt{|w|}}\quad\text{if}\,\,\,  |w|>r,\,|{\rm Im\,} w|\ge 1.
\end{equation*}
Since the function $\,w\mapsto \sqrt{|w|}\exp(-|{\rm Im\,} w|) J_\nu(w)\,$ is continuous and zero free, away from the
real axis, its modulus must have a positive minimum on $\,\left\{w: |w|\le r, \, |{\rm Im\,} w| \ge 1\right\}.$
Thus there exists a constant $\,d_\nu>0\,$ such that
\begin{equation}\label{EL1}
\left|J_\mu(w)\right| \ge d_\nu  \frac{\exp(|{\rm Im\,} w|)}{\sqrt{|w|}}\quad\text{if}\,\,\, |{\rm Im\,} w|\ge 1.
\end{equation}

If $\,\nu\ge -1/2,$ then Lemma \ref{lemmaP1} shows that
$$\left| J_\nu(wt) \sqrt{wt}\right|\le c_\nu  \exp(t|{\rm Im\,} w|)\le  c_\nu  \exp(|{\rm Im\,} w|)$$
for all $\,w\ne 0\,$ and for all $\,0<t<1\,$ and hence \eqref{BE3} follows at once by \eqref{EL1}.
If $\,-1<\nu<-1/2,$ then Lemma \ref{lemmaP1} gives
\begin{align*}
\left| J_\nu(wt) \sqrt{wt}\right| &\le c_\nu  \exp(t|{\rm Im\,}w|)\left[ 1+ (|w|t)^{\nu + 1/2}\right] \\
&\le c_\nu' \exp(|{\rm Im\,} w|)\, t^{\nu+1/2}
\end{align*}
for all $w$ with$\,|w|\ge 1\,$ and for all $\,0<t<1,$ where $c_\nu'$ denotes another constant, and hence
\eqref{BE3} follows immediately in this case, too.
\end{proof}

Let us denote by $U_{n, Y}$ the upper side of rectangle $R_n$ and assume first that
$\,\nu\ge -1/2.$ For $\,w\in U_{n, Y},$ we have $\,|w|\ge Y>2|z|\,$
and if we further assume $\,Y>1,$ then we may apply the estimate \eqref{BE3} to obtain
$$|\Psi(w, t)| \le \frac{A|w|^{1/2}}{|w|^\lambda (|w|-|z|)} \le \frac{2A}{Y^{\nu-\mu+2}},$$
where we have used the condition $\,\nu-\mu+2>0.$ As a consequence,
\begin{align*}
\left|\int_{U_{n, Y}} \psi(w) dw\right| &\le \int_0^1\int_{U_{n, Y}} |f(t)|\left|\Psi(w, t)\right|  |dw|dt\\
&\le \frac{4A X_n}{\,Y^{\nu-\mu+2\,}}\int_0^1 |f(t)|dt,
\end{align*}
which shows in effect that if $n$ is fixed, then
$$\lim_{Y\to\infty} \left|\int_{U_{n, Y}} \psi(w) dw\right| =0.$$

In the case $\,-1<\nu<-1/2,$ \eqref{BE3} gives the alternative bound
\begin{align*}
\left|\int_{U_{n, Y}} \psi(w) dw\right| \le \frac{4A X_n}{\,Y^{\nu-\mu+2}\,}\int_0^1 t^{\nu+1/2} |f(t)|dt,
\end{align*}
whence the same conclusion remains valid. By dealing with the integral
along the lower side of $R_n$ in a similar way, we conclude that
\begin{align}\label{Pf4}
&\frac{1}{2\pi i}\left(\int_{ X_n -\infty i}^{X_n + \infty i} -
\int_{-X_n -\infty i}^{-X_n +\infty i}\right)\psi(w) dw \nonumber\\
&\qquad=  \frac{\mathcal{H}_\nu(f)\rb{z}}{z^{\lambda}J_\mu (z)} - \frac{b}{z}
+\sum_{m=1}^n a_m \rb{ \frac{1}{z-j_{\mu,m}} + \frac{1}{z+j_{\mu,m}}}.
\end{align}

\medskip

{\bf 3.} We next consider taking limits $\,n\to\infty\,$ on both sides of \eqref{Pf4}.
On transforming $\,w\to e^{\pi i}w,$ the same reasonings used to derive \eqref{Pf2} gives
$$\int_{-X_n -\infty i}^{-X_n +\infty i}\Psi(w, t) dw = \int_{X_n -\infty i}^{X_n +\infty i}
\frac{J_\nu(wt) \sqrt{wt}}{\,w^\lambda (w+z) J_\mu(w)\,}dw$$
and hence we obtain
\begin{align}\label{Pf5}
&\frac{1}{2\pi i}\left(\int_{ X_n -\infty i}^{X_n + \infty i} -\int_{-X_n -\infty i}^{-X_n +\infty i}\right)\psi(w) dw\nonumber\\
&\qquad= \frac{z}{\pi i}\int_0^1\int_{ X_n -\infty i}^{X_n + \infty i} f(t) \left[ \frac{J_\nu(wt) \sqrt{wt}}{\,w^\lambda (w^2 - z^2) J_\mu(w)\,}\right] dw dt.
\end{align}

As an alternative of \eqref{BE3}, we claim that there exist an integer $\,n_1\ge n_0\,$
and a constant $\,B>0,$ depending on $\mu, \nu$ but independent of $n$, such that
\begin{align}\label{BE4}
\left|\frac{\,J_\nu(wt) \sqrt{wt}\,}{J_\mu(w)}\right|  \le B\,|w|^{1/2}\cdot\left\{\begin{aligned}
&{\quad 1\qquad \,\,\,\big(\nu\ge - 1/2\big),}\\
&{t^{\nu+1/2}\,\,\quad\,\big(\nu<-1/2\big)} \end{aligned}\right.
\end{align}
for all $w$ with $\,{\rm Re}\,w = X_n,\,\,n\ge n_1\,$ and for all $\,0<t<1.$

Indeed, since $X_n$ was chosen so as to
\begin{align*}
\left|\cos\left(w-\frac{\mu\pi}{2}-\frac{\pi}{4}\right)\right| = \cosh\,({\rm Im}\,w ),
\end{align*}
if we apply Hankel's asymptotic formula \eqref{ZH} to choose $\,n_1\ge n_0\,$ so that
\begin{equation*}
\left|J_\mu(w)\right| \ge \frac 16  \frac{\cosh\,({\rm Im\,} w)}{\sqrt{|w|}}
\end{equation*}
for all $w$ with $\,{\rm Re}\,w = X_n,\,\,n\ge n_1,$ and Lemma \ref{lemmaP1},
\eqref{BE4} follows along the same scheme as employed in the proof of Lemma \ref{lemmaP3}.

Let $\,\nu\ge -1/2.$ As easily verified by \eqref{BE4}
and the assumption $\,X_n>2|z|,$ if $\,n\ge n_1,$ then
the modulus of the right side of \eqref{Pf5} does not exceed
\begin{align}\label{Pf6}
&\frac{4B}{3\pi} \int_0^1|f(t)| dt \int_{X_n -\infty i}^{X_n + \infty i} |w|^{-(\nu-\mu+3)} |dw|\nonumber\\
&\qquad = \frac{8B}{3\pi} \int_0^1|f(t)| dt \int_0^\infty \frac{ds}{\,|X_n + is|^{\nu-\mu+3}\,}.
\end{align}

By changing variables $\,s\to X_n \tan\theta,$ it is simple to evaluate
\begin{align*}
\int_0^\infty \frac{ds}{\,|X_n + is|^{\nu-\mu+3}\,} &= \frac{1}{X_n^{\nu-\mu+2}}
\int_0^{\pi/2} \cos^{\nu-\mu+1}\theta\,d\theta\\
&= \frac{\Gamma\left(\frac{\nu-\mu+2}{2}\right)\Gamma(1/2)}{2\, \Gamma\left(\frac{\nu-\mu+3}{2}\right)}
\frac{1}{X_n^{\nu-\mu+2}}\,.
\end{align*}
Consequently, the right side of \eqref{Pf6} is bounded by
$$ \frac{\,C\,}{X_n^{\nu-\mu+2}}\int_0^1 |f(t)|dt$$
for some constant $\,C = C(\mu, \nu)>0\,$ and we may conclude that
$$
\lim_{n\to\infty}\left|\frac{1}{2\pi i}\left(\int_{ X_n -\infty i}^{X_n + \infty i} -\int_{-X_n -\infty i}^{-X_n +\infty i}\right)\psi(w) dw\right| =0.$$
Passing to the limit in the identity \eqref{Pf4}, we obtain the desired partial
fraction expansion formula, provided that the series \eqref{Paf1} converges.

In the case $\,-1<\nu<-1/2,$ an obvious modification gives the bound 
$$ \frac{\,C \,}{X_n^{\nu-\mu+2}}\int_0^1 t^{\nu+1/2} |f(t)|dt$$
for the right side of \eqref{Pf6}, which leads to the same conclusion.

\medskip

{\bf 4.} We shall now prove that the series of \eqref{Paf1} converges
absolutely for each fixed $\,z\in\mathcal{D}_\mu\,$ and the convergence is uniform
on each fixed compact subset of $\mathcal{D_\mu}$, which will complete the proof.

For $\,\nu\ge -1/2,$ Lemma \ref{lemmaP1} gives the uniform boundedness
\begin{align}\label{UB}
\left| \mathcal{H}_\nu(f) (j_{\mu,m})\right|
&\le \int_0^1 |f(t)| \left| J_\nu(j_{\mu, m} t) \sqrt{j_{\mu, m} t}\right| dt\nonumber\\
&\le c_\nu \int_0^1 |f(t)|dt, \quad m=1,2, \cdots.
\end{align}
Concerning the coefficient of \eqref{Paf1}, it is thus evident by Lemma \ref{lemmaP2} that there exists
a positive constant $c_{\mu, \nu}$, independent of $m$, such that
\begin{align}\label{UBS1}
    \left|a_m\right|  &= \left|\frac{ \mathcal{H}_\nu (f)\rb{j_{\mu,m}} }{\,j_{\mu,m}^{\nu-\mu+3/2}
 J_{\mu+1}\rb{j_{\mu,m}}}\right|\nonumber\\
    &\le c_{\mu, \nu}\,j_{\mu, m}^{\,-(\nu-\mu+1)}  \int_0^1\vb{f(t)}dt,\quad m=1, 2, \cdots.
\end{align}

By the interlacing of zeros of $\,J_\mu(z), J_{\mu+2}(z)\,$ and the monotonicity of
$j_{\mu, m}$ with respect to $\,\mu>-1,$ for each fixed $m$, we observe that
\begin{equation}\label{UBS2}
 j_{\mu, m} > j_{\mu+2, m-1}>j_{1/2, m-1} = (m-1)\pi,\quad m=2, 3, \cdots.
 \end{equation}

 \begin{itemize}
 \item[(i)] We fix a point $\,z_0\in\mathcal{D}_\mu\,$ and choose an integer $\,m_0\ge 2\,$ such that
 $\,(m_0-1)\pi>2 |z_0|.$ If $\,m\ge m_0,$ then \eqref{UBS2} indicates that $\,j_{\mu, m}>2|z_0|\,$ and thus
 $\,\left| z_0^2 - j^2_{\mu,m}\right|> 3j^2_{\mu,m}/4,$ which in turn implies
  $$\left|\frac{1}{z_0-j_{\mu,m}} + \frac{1}{z_0 +j_{\mu,m}}\right|=
\left|\frac{2z_0}{z_0^2-j^2_{\mu,m}}\right| \le \frac{8|z_0|}{\,3 j_{\mu, m}^2\,}.$$
It follows from \eqref{UBS1} and \eqref{UBS2} that
\begin{align}\label{LE1}
    &\sum_{m= m_0}^\infty \left|a_m \left(\frac{1}{z_0 -j_{\mu,m}} + \frac{1}{z_0 +j_{\mu,m}}\right)\right| \nonumber\\
    &\quad \le \frac{8|z_0| c_{\mu, \nu}}{3} \int_0^1\vb{f(t)}dt\sum_{m=m_0}^\infty\,\frac{1}{j_{\mu,m}^{\nu-\mu+3}}\nonumber\\
     &\quad \le \frac{8|z_0| c_{\mu, \nu}}{3} \int_0^1\vb{f(t)}dt\sum_{m=m_0}^\infty\,\frac{1}{[(m-1)\pi]^{\nu-\mu+3}},
   \end{align}
which shows in effect, due to the condition $\,\nu-\mu+3>1,$ that the series defined in \eqref{Paf1}
converges absolutely at $z_0$.

\item[(ii)] Let us fix a compact set $\,\mathcal{E}\subset \mathcal{D}_\mu.$ Setting $\,\rho= \sup\left\{|z| : z\in\mathcal{E}\right\},$
we choose an integer $\,\ell\ge 2\,$ with $\,(\ell-1)\pi>2\rho.$ If $\,m>\ell\,$ and $\,z\in\mathcal{E},$ then
\eqref{UBS2} implies that $\,|z|\le \rho <j_{\mu, m}/2\,$ and so
$$\left|\frac{1}{z-j_{\mu,m}} + \frac{1}{z+j_{\mu,m}}\right|=
\left|\frac{2z}{z^2-j^2_{\mu,m}}\right| \le \frac{8\rho}{\,3 j_{\mu, m}^2\,}.$$

Let $\mathcal{E}^c$ denote the complement of $\mathcal{E}$. Since $\mathcal{E}^c$ is an open set containing $\,\pm j_{\mu, m}\,$ for all $m$,
we can find a constant $\,\delta>0\,$ so that
$$\,\Big\{ w\in\mathbb{C} : \left |w\pm j_{\mu, m}\right|<\delta\Big\}\subset \mathcal{E}^c,\quad 1\le m\le  \ell,$$
which implies at once 
$$\left|\frac{1}{z-j_{\mu,m}} + \frac{1}{z+j_{\mu,m}}\right| \le \frac{2}{\delta}\quad\text{for all}\,\,\, z\in\mathcal{E},\, \,1\le m\le \ell.$$

By combining with \eqref{UBS1}, we deduce the uniform bounds
\begin{align}\label{UBS3}
&\sup_{z\in\mathcal{E}}\left|a_m \left(\frac{1}{z-j_{\mu,m}} + \frac{1}{z+j_{\mu,m}}\right)\right| \le  A_m c_{\mu, \nu} \int_0^1\vb{f(t)}dt,\\
&\quad\text{where}\quad A_m =\left\{\begin{aligned}
&{\,\,\frac{2}{\,\delta j_{\mu, m}^{\nu-\mu+1}\,}\quad\text{for}\quad 1\le m\le \ell,}\\
&{\,\,\,\frac{8\rho}{\,3 j_{\mu, m}^{\nu-\mu+3}\,}\quad\text{for}\quad  m\ge\ell +1.}
\end{aligned}\right.\nonumber
\end{align}

Due to the condition $\,\nu-\mu+3>1,$ it is evident by \eqref{UBS2} that 
\begin{align*}
    \sum_{m=1}^\infty A_m &= \sum_{m=1}^\ell \frac{2}{\,\delta  j_{\mu, m}^{\nu-\mu+1}\,}
    + \sum_{m=\ell+1}^\infty \frac{8\rho}{\,3 j_{\mu, m}^{\nu-\mu+3}\,}\nonumber\\
    &\le \frac{2}{\delta} \sum_{m=1}^\ell \frac{1}{\, j_{\mu, m}^{\nu-\mu+1}\,}
    + \frac{8\rho}{3}\sum_{m=\ell+1}^\infty \frac{1}{[(m-1)\pi]^{\nu-\mu+3}} \nonumber\\
    &<\infty.
   \end{align*}
Therefore we may conclude that
the series defined in \eqref{Paf1} converges uniformly on $\mathcal{E}$ by the Weierstrass M-test.
\end{itemize}

In the case $\,-1<\nu<-1/2,$  Lemma \ref{lemmaP1} gives the alternative
\begin{equation}\label{UB2}
\left| \mathcal{H}_\nu(f) (j_{\mu,m})\right|
\le c_\nu \int_0^1 t^{\nu+1/2} |f(t)|dt, \quad m=1,2, \cdots,
\end{equation}
in place of  \eqref{UB}, where $c_\nu$ may be a different constant. Proceeding along the same lines as above
and adjusting multiplicative constants in an obvious way, it is not difficult to confirm that estimates \eqref{LE1}, \eqref{UBS3}
remain valid with $\int_0^1 |f(t)|dt$ replaced by
$\int_0^1 t^{\nu+1/2} |f(t)|dt$. We conclude that the series defined in \eqref{Paf1} converges absolutely at
$\,z_0\in \mathcal{D}_\mu\,$ and uniformly on $\mathcal{E}$.

Theorem \ref{theoremP1} is now completely proved.\qed

\subsection{Equivalent forms}
By using \eqref{H3}, it is immediate to express the partial fraction expansion formula
\eqref{Paf1} in terms of the Hankel transform defined by \eqref{H2}.

\smallskip

\begin{corollary} Under the same assumptions of Theorem \ref{theoremP1}, we have
\begin{align}\label{Paf3}
 \frac{\,\bH_\nu(f)\rb{z}\,}{z \bJ_\mu (z)}  &=\frac{\,\bH_\nu(f)(0)\,}{z} -2(\mu+1)
 \sum_{m=1}^\infty \frac{\bH_\nu(f)\rb{j_{\mu, m}}}{\,j_{\mu, m}^2\bJ_{\mu+1}\rb{j_{\mu, m}}\,}\nonumber\\
 &\qquad\times\,\,\rb{ \frac{1}{z-j_{\mu,m}} + \frac{1}{z+j_{\mu,m}}}\qquad(z\in \mathcal{D}_\mu).
 \end{align}
\end{corollary}

A remarkable consequence is that the image of Hankel transform of order $\nu$ can be recovered fully from its
sampled values at the positive zeros of Bessel function $J_\mu(z)$ of any order $\,-1<\mu<\nu+2.$
To be precise, if we multiply \eqref{Paf3} by $z$ and rewrite the coefficients with the aid of identity
\begin{equation}\label{W2}
\bJ_\mu'(z) = -\frac{z}{2(\mu+1)}\,\bJ_{\mu+1}(z),
\end{equation}
we obtain an $L^1$-version of \emph{sampling theorems} for Hankel transforms.

\begin{corollary} Let $\,\nu>-1,\,-1<\mu<\nu+2.$ For an integrable function $f(t)$ satisfying  \eqref{IA},
put $\,\phi(z) =\bH_\nu(f)(z).$ Then for all $\,z\in\mathbb{C},$
\begin{align*}\label{Sap1}
\phi(z) =\bJ_\mu(z) \biggl[ \phi(0) + 2z^2
\sum_{m=1}^\infty \frac{\phi\rb{j_{\mu, m}}}{\,j_{\mu, m}\,\bJ_{\mu}'\rb{j_{\mu, m}}\,}\,
\frac{1}{z^2-j_{\mu,m}^2}\biggr],
\end{align*}
where the series converges uniformly on each compact
subset of $\mathbb{C}$.
\end{corollary}

Since the singularities at $\pm j_{\mu, m}$ are easily seen to be removable, 
the expression on the right represents an entire function. In the special case $\,\mu=1/2,$
by using the known formula \eqref{SJ}, it is straightforward to find that
the above series representation reduces to
$$ \phi(z) = \sum_{m\in\mathbb{Z}} \phi(m\pi)\, \frac{\sin(z-m\pi)}{\,z-m\pi\,}\quad(z\in\mathbb{C},\,\nu>-1).$$

As an illustration, if we take
\begin{equation}\label{EX1}
f(t) = \frac{2}{B(\lambda-\nu, \nu+1)} \,(1-t^2)^{\lambda-\nu-1} t^{\nu+1/2},\quad 0<t<1,
\end{equation}
where $B$ denotes the beta function and $\,\lambda>\nu>-1,$ it is routine to compute
$\,\bH_\nu(f)(z) = \bJ_\lambda(z)\,$ and the above representations give
\begin{align*}
\bJ_\lambda(z) &=\bJ_\mu(z) \biggl[ 1 + 2z^2
\sum_{m=1}^\infty \frac{\bJ_\lambda\rb{j_{\mu, m}}}{\,j_{\mu, m}\,\bJ_{\mu}'\rb{j_{\mu, m}}\,}\,
\frac{1}{z^2-j_{\mu,m}^2}\biggr]\\
&= \sum_{m\in\mathbb{Z}} \bJ_\lambda(m\pi)\, \frac{\sin(z-m\pi)}{\,z-m\pi\,}\qquad(z\in\mathbb{C}),
\end{align*}
valid for $\,\lambda>-1,\,-1<\mu<\lambda+2.$ We remark that
the former extends Higgins \cite[(7)]{H1} and
the latter is equivalent to \cite[\S 19.4, (7)]{Wa}
obtained by Nagaoka \cite{Na} while studying the intensity of diffracted light
on a cylindrical surface. For $\,\lambda>0,$ it is simple to find that
$\bJ_\lambda(z)$ is band-limited and hence the latter is also obtainable from
Shannon's sampling theorem (see \cite{H1}, \cite{H2}, \cite{J} for more details and applications).

\section{Hurwitz-P\'olya theorems}
This section aims to revisit some of the results obtained by Hurwitz and P\'olya \cite{P} on Fourier transforms
in terms of Hankel transforms for the sake of completeness and subsequent applications.

Concerning Fourier transforms $U(z), V(z)$ as defined in \eqref{HP1},
if we apply the partial fraction expansion formula \eqref{Paf3} with $\,\mu = \pm 1/2,$
it is immediate to derive \eqref{HP2}, \eqref{HP3} and the additional expansions
\begin{equation}\label{HP4}
\frac{\,U(z)\,}{\,\sin z\,} = \frac{U(0)}{z} +
\sum_{m=1}^\infty (-1)^m U(m\pi)\left(\frac{1}{z-m\pi} + \frac{1}{z+ m\pi}\right),
\end{equation}
\begin{align}\label{HP5}
\frac{V(z)}{z\cos z} &= V'(0) + z\sum_{m=1}^\infty\frac{(-1)^m V\left[(m-1/2)\pi\right]}{[(m-1/2)\pi]^2}\nonumber\\
&\qquad\qquad\quad\times\left[\frac{1}{z-(m-1/2)\pi} + \frac{1}{z+ (m-1/2)\pi}\right].
\end{align}
valid under the assumption that $f(t)$ is integrable.

We remark that \eqref{HP4} coincides with \cite[(38)]{P}  but \eqref{HP5} is not listed in \cite{P}. Thus \eqref{Paf3} recovers
all expansions obtained by Hurwitz and P\'olya, with the additional formula \eqref{HP5},
as special cases. Since the choice of $\mu$ is free in the range $\,-1<\mu<\nu+2,$
more expansions are also available.

As discovered by Hurwitz and developed formally by P\'olya, it turns out that
these partial fraction expansions provide effective means of investigating
the existence and nature of zeros as well as their distributions.

\begin{theorem}\label{lemmaP}{\rm{(Hurwitz and P\'olya \cite{P})}}
Let $G(z)$ be an even entire function having zeros only at $\,\pm \alpha_m,\,0<\alpha_1<\alpha_2<\cdots.$
Suppose that $F(z)$ is an even entire function such that $\,\left|F(\alpha_m)\right|>0\,$
for all $m$ and
\begin{equation}\label{HP}
\frac{F(z)}{\,z G(z)\,} = \frac{b}{z} + \sum_{m=1}^\infty c_m\left(\frac{1}{z-\alpha_m} +\frac{1}{z+\alpha_m}\right)
\end{equation}
for all $\,z\in\mathbb{C}\setminus\left\{0, \,\pm \alpha_1, \,\pm \alpha_2, \cdots\right\},$
where $\,b\in\R\,$ and $\,c_m>0\,$ for each $m$.
\begin{itemize}
\item[\rm(i)] If $\,b>0,$ then $F(z)$ has an infinity of zeros which are all real and
it has exactly one zero in each of the intervals
$$\big(0,\,\alpha_1\big),\,\,\big(\alpha_m,\,\alpha_{m+1}\big),\quad m= 1, 2, \cdots,$$
and no positive zeros elsewhere.
\item[\rm(ii)] If $\,b<0\,$ and $F(z)$ is subject to the additional assumption
$\,|F(iy)|>0\,$ for all $\,y\in\mathbb{R},$ then $F(z)$ has an infinity of zeros which are all real and
it has exactly one zero in each of the intervals
$$\big(\alpha_m,\,\alpha_{m+1}\big),\quad m= 1, 2, \cdots,$$
and no positive zeros elsewhere.
\end{itemize}
\end{theorem}

This important result is an abstract formulation of what Hurwitz and P\'olya observed
(\cite[\S 5]{P}). To reproduce their ideas of proof in the present setting,
let us consider the $N$th partial sum of \eqref{HP}
\begin{align*}
S_N(z) &= \frac{b}{z} + \sum_{m=1}^N c_m\left(\frac{1}{z-\alpha_m} +\frac{1}{z+\alpha_m}\right)\\
&= \frac{P_{2N}(z)}{\,z\left(z^2-\alpha_1^2\right)\cdots \left(z^2- \alpha_N^2\right)\,},
\end{align*}
where $P_{2N}(z)$ is an even polynomial of degree $2N$.
As $\,S_N(x),\,x>0,$ takes every real value in each open interval between
two consecutive poles in a monotonically decreasing manner, it is evident that
$S_N(z)$ has exactly one positive zero in $\,\big(\alpha_m,\,\alpha_{m+1}\big),\,m= 1, 2, \cdots,N-1,$
whatever $b$ is.

If $\,b>0,$ then $S_N(z)$ has exactly one positive zero in
$\,\big(0, \,\alpha_1\big)\,$ by the same reason and thus all of $2N$ zeros of $S_N(z)$
are located in view of symmetry.

In the case $\,b<0,$ if we consider the asymptotic behavior
\begin{align*}
S_N(z) &= \frac{1}{z} \sum_{m=1}^N \frac{\,(b+2c_m) z^2 -b \alpha_m^2\,}{z^2 - \alpha_m^2}\\
&\sim \frac{\,b + 2(c_1 +\cdots + c_{N})\,}{z}\quad\text{as $\,|z|\to\infty,$}
\end{align*}
we infer that $S_N(z)$ has one positive zero in $\,\big(\alpha_N,\,\infty\big)\,$ only when
$$b + 2(c_1 +\cdots + c_{N})<0.$$
Otherwise, $S_N(z)$ has two complex zeros which must be purely imaginary because $P_{2N}(z)$ is even.
Consequently, $S_N(z)$ has either $2N$ real zeros or $2N-2$ real zeros and two purely imaginary
zeros.

If we transfer the above analyses to the function $\,F(z)/zG(z)$ by letting $\,N\to\infty,$
since the possibility that the zeros move towards
$\alpha_m$'s are excluded due to the hypothesis $\,\left|F(\alpha_m)\right|>0\,$
for all $m$, we may conclude that if $\,b>0,$ then
$F(z)$ has only real zeros distributed in the pattern of (i) and if $\,b<0,$ then
$F(z)$ has real zeros distributed in the pattern of (ii) and possibly two purely imaginary zeros.
Since it is assumed that $\,|F(iy)|>0\,$ for all $\,y\in\mathbb{R},$
$F(z)$ can not have purely imaginary zeros and the assertion follows.

\begin{remark}\label{remarkHP}
Without assuming that
$\,\left|F(\alpha_m)\right|>0\,$ for all $m$,
an inspection on the proof indicates that $F(z)$ still has
an infinity of zeros but no complex zeros. Since the zeros of $S_N(z)$ could move towards
the end-points in the limiting process, each $\alpha_m$ could be a zero of $F(z)$
and it is not difficult to infer that $F(z)$ has at most two zeros in each interval
$\,\left(0,\,\alpha_1\right],\,\left(\alpha_m,\,\alpha_{m+1}\right]\,$
in the case (i) and $\,\left[\alpha_m,\,\alpha_{m+1}\right)\,$
in the case (ii), where $\, m=1,2,\cdots.$
\end{remark}

\section{Zeros of Hankel transforms}
By applying the preceding Hurwitz-P\'olya theorem on the basis of
partial fraction expansion formula \eqref{Paf3}, we
are about to obtain information on the zeros of Hankel transforms.
Before proceeding further, let us state a simple consequence of \eqref{Paf3} which will be
used to determine the multiplicity of each zero. As usual, we shall denote
the Wronskian of $f(z), g(z)$ by
$$W[f, g](z) =  f(z)g'(z)-f'(z) g(z).$$

\begin{lemma}
Under the same assumptions of Theorem \ref{theoremP1}, we have
\begin{align}\label{W1}
W\big[\bJ_\mu, \bH_\nu (f)\big](z)
=8z(\mu+1)\bJ_\mu^2(z)\sum_{m=1}^\infty
\frac{\,\bH_\nu(f)\rb{j_{\mu, m}}\,}{\bJ_{\mu+1}\rb{j_{\mu, m}}}\,\frac{1}{\rb{z^2 - j_{\mu, m}^2}^2},
\end{align}
valid for all $\,z\in\mathbb{C}.$
\end{lemma}

If we multiply by $z$ and differentiate both sides of \eqref{Paf3},
where termwise differentiation is permissible due to the uniform convergence
of the derived series on every compact subset of $\mathcal{D}_\mu$, we obtain
the Wronskian formula \eqref{W1} on $\mathcal{D}_\mu$.
Since the singularities at $\pm j_{\mu, m}$ are easily seen to be removable,
the formula remains valid for all $\,z\in\mathbb{C}.$

\begin{theorem}\label{theoremZ1}
 Let $\,\nu>-1,\,-1<\mu<\nu+2\,$ and $f(t)$ be a positive integrable function, defined for $\,0<t<1,$
 subject to the condition \eqref{IA}. Put
 \begin{equation*}
\sigma_m = (-1)^{m+1}\bH_\nu(f) \rb{j_{\mu, m}},
\quad m=1,2,\cdots.
\end{equation*}
If $\rb{\sigma_m}$ keeps constant sign for all $m$, then $\bH_\nu(f)(z) $
has only an infinity of real simple zeros whose positive zeros
are distributed as follows.
\begin{itemize}
\item[\rm(i)] If $\,\sigma_m>0\,$ for all $m$, then each of the intervals
$$\,\big(j_{\mu, m}, \, j_{\mu, m+1}\big),\quad m=1,2,\cdots,$$
contains exactly one zero with no positive zeros elsewhere.
\item[\rm(ii)] If $\,\sigma_m<0\,$ for all $m$, then each of the intervals
$$\,\big(0,\,j_{\mu, 1}\big),\,\,\big(j_{\mu, m}, \, j_{\mu, m+1}\big),\quad m=1,2,\cdots,$$
contains exactly one zero with no positive zeros elsewhere.
\end{itemize}
\end{theorem}

For the proof, we first note that $\bH_\nu(f)(z) $ is entire with
\begin{align*}
\bH_\nu(f)(iy) = \sum_{m=0}^\infty \left(\int_0^1 t^{2m+\nu+1/2} f(t) dt\right)
\frac{(y/2)^{2m}}{\,m!\, (\nu+1)_m\,}>0
\end{align*}
for all $\,y\in\R\,$ and $\bH_\nu(f)(z)$ is even.
Owing to the inequalities (\cite[\S 15.22]{Wa})
$$0<j_{\mu, 1}<j_{\mu+1, 1}<j_{\mu, 2}<j_{\mu+1, 2}<\cdots\quad(\mu>-1)$$
and the positivity $\,\bJ_{\mu+1}(x)>0\,$ for $\,0<x<j_{\mu+1, 1},$ we have
$${\rm{sgn}}\,\Big[\bJ_{\mu+1} \rb{j_{\mu, m}}\Big] = (-1)^{m+1},\quad m =1, 2, \cdots.$$

In the case (i), if we notice the inequalities
$$b \equiv -\int_0^1 t^{\nu+1/2} f(t) dt <0,
\quad c_m \equiv 2(\mu+1)\frac{\bH_\nu(f)\rb{j_{\mu, m}}}{\,j_{\mu, m}^2\bJ_{\mu+1}\rb{j_{\mu, m}}\,}>0
$$
for each $m$ and rewrite formula \eqref{Paf3} into the form
\begin{align*}
- \frac{\bH_\nu(f)\rb{z}}{z \bJ_\mu (z)}  =\frac{b}{z} +
\sum_{m=1}^\infty c_m\left(\frac{1}{z- j_{\mu, m}} +\frac{1}{z+j_{\mu, m}}\right),
\end{align*}
then the stated results are immediate consequences of Theorem \ref{lemmaP} except the simplicity of zeros.
Similarly, the results in the case (ii) can be proved, except the simplicity of zeros, with obvious modifications.

Regarding the simplicity of each zero, we observe that Wronskian formula
\eqref{W1} implies the inequalities $\,W\big[\bJ_\mu, \bH_\nu (f)\big](x) >0\,$ for all $\,x>0\,$ in the case (i) and
$\,W\big[\bJ_\mu, \bH_\nu (f)\big](x) <0\,$ for all $\,x>0\,$ in the case (ii). If
$$\bH_\nu(f)(\widehat{x}) = 0 = \bH_\nu(f)'(\widehat{x})\quad\text{for some}\quad \widehat{x}>0,$$
then the Wronskian of $\bJ_\mu, \bH_\nu (f)$ vanishes at $\widehat{x}$, which leads to a contradiction in any case.
It is thus shown that any positive zero of $\bH_\nu(f)(x)$ must be simple and
Theorem \ref{theoremZ1} is now completely proved.\qed

\section{Laguerre-P\'olya class}
Our purpose here is to identify the Hankel transform as an
entire function of the Laguerre-P\'olya class under the same conditions for sign changes.

\medskip

{\bf 1.}
For an integrable function $f(t)$ satisfying \eqref{IA},
it is elementary to find by integrating termwise
that $\bH_\nu(f)(z)$ is an even entire function with
\begin{align}\label{F0}
\left\{\begin{aligned}
\bH_\nu(f)(z) &{\,\,= \,\sum_{m=0}^\infty \frac{\, \beta_m(f)\,}{\,m!\,(\nu+1)_m\,}\left(-\frac{z^2}{4}\right)^{m},\quad\text{where}}\\
\beta_m(f)\,\,  &{\,\,=\, \int_0^1 t^{2m+\nu+1/2}f(t)dt,\quad m=0, 1, 2, \cdots.}\end{aligned}\right.
\end{align}

Concerning the (growth) order of $\bH_\nu(f)(z)$, let
$$ b_m = \frac{(-1)^m \beta_m(f)}{4^m m! (\nu+1)_m}, \quad m=0, 1,\cdots,$$
the coefficient of $z^{2m}$. By using generalized Stirling's formula \cite[5.11.7]{dlmf}
$\,\Gamma(m+ c) \,\sim\, \sqrt{2\pi}\,e^{-m} m^{m+c-1/2}\,$ as $\,m\to\infty,$
where $c$ is any real number, and $\,(\nu+1)_m = \Gamma(m+\nu+1)/\Gamma(\nu+1),$
it is easy to see that
$$\log (1/|b_m|) \,\sim\, 2m\log m \quad\text{as $\,m\to\infty\,$}$$
and consequently $\bH_\nu(f)(z)$ has the order
\begin{equation*}\label{order}
\limsup \frac{2m\log 2m}{\,\log (1/ |b_m|)\,\,}  = 1
\end{equation*}
(see \cite{Le} for the definition and  \cite{BK} for analogous arguments).

\medskip

{\bf 2.} Suppose now that $f(t)$ satisfies the additional assumptions of Theorem \ref{theoremZ1} so that
$\bH_\nu(f)(z)$ becomes a real entire function having only
an infinity of real simple zeros. Let
$\big(\zeta_{m}\big)$ be the sequence of all positive zeros arranged in ascending order
of magnitude. Since the positive zeros of $\bH_\nu(f)(z), J_\mu(z)$ are shown to be interlaced and
$$\sum_{m=1}^\infty \frac{1}{j_{\mu, m}} = +\infty,
\quad \sum_{m=1}^\infty \frac{1}{j_{\mu, m}^2} = \frac{1}{4(\mu+1)},$$
we find that
\begin{equation*}\label{F2}
\sum_{m=1}^\infty \frac{1}{\zeta_{m}} = +\infty\quad\text{but}\quad
\quad \sum_{m=1}^\infty \frac{1}{\zeta_{m}^2}<\infty.
\end{equation*}

In view of Hadamard's factorization theorem (\cite[\S 4.2]{Le}), it is thus shown that the Hankel transform $\bH_\nu(f)(z)$ belongs to the
Laguerre-P\'olya class, to be denoted by $\mathcal{LP}$ hereafter, and has genus one.

\begin{theorem}\label{theoremLP}
Let $\,\nu>-1,\,-1<\mu<\nu+2.$ If $f(t)$ is a positive integrable function satisfying \eqref{IA}
such that the sequence $\,\big\{\bH_\nu(f) \rb{j_{\mu, m}}\big\}\,$ alternates in sign for $\, m=1,2,\cdots,$ then
$\,\bH_\nu(f)\in \mathcal{LP}\,$ with
\begin{equation}\label{F3}
\bH_\nu(f)(z) = \bH_\nu(f)(0)\prod_{m=1}^\infty \left(1-\frac{z^2}{\zeta_m^2}\right).
\end{equation}
\end{theorem}

While the theorem is of considerable interest in the theory of entire functions,
we present an application of \eqref{F3} after
Euler and Rayleigh \cite[\S 15.5, 15.51]{Wa}. By differentiating
logarithmically and expanding each term in geometric series,  it is immediate to deduce the relation
$$ -\frac{d}{dz} \bH_\nu(f)(\sqrt z)
= \bH_\nu(f)(\sqrt z) \sum_{k=0}^\infty \Delta_k z^k,\quad \Delta_k :=
\sum_{m=1}^\infty \frac{1}{\,\zeta_m^{2k+2}\,},$$
valid for all $\,|z|<\zeta_1.$ If we expand both sides into power series,
with the aid of \eqref{F0} and the Cauchy product formula, and equate the coefficients, we can
compute $(\Delta_k)$ explicitly in an inductive manner.
For example,
\begin{align}
\sum_{m=1}^\infty \frac{1}{\,\zeta_m^2\,} &= \frac {\beta_1(f)}{\,4(\nu+1) \beta_0(f)\,},\label{F4}\\
\sum_{m=1}^\infty \frac{1}{\,\zeta_m^4\,} &= \frac{\,(\nu+2) [\beta_1(f)]^2 - (\nu+1) \beta_0(f)\beta_2(f)\,}
{\,16(\nu+1)^2 (\nu+2) \left[\beta_0(f)\right]^2\,}.
\end{align}

We refer to \cite{GM}, \cite{IM}, \cite{Sn2} and further references therein for the related sums of zeros
of Bessel functions and various applications.

\medskip
{\bf 3.} It is well known that each entire function of the Laguerre-P\'olya class $\mathcal{LP}$
arises as the uniform limit of real polynomials having only real zeros. By Rolle's theorem,
hence, it is simple to prove that $\mathcal{LP}$ is closed under differentiation.
By exploiting the method of partial fractions, we shall now prove that certain subclass of $\mathcal{LP}$
is closed under more general differential operators, which will be useful in later applications.

\begin{theorem}\label{lemmaLP}
Let $G(z)$ be an even entire function of the form
\begin{equation}\label{typeLP}
G(z) = \sum_{m=0}^\infty \frac{\,(-1)^m\gamma_m\,}{m!}\,z^{2m}\qquad(z\in\mathbb{C}),
\end{equation}
where $(\gamma_m)$ is a sequence of non-zero reals subject to the condition
\begin{equation}\label{condLP}
\gamma_m = O\left(\frac{1}{m!}\right)\quad\text{as $\,m\to\infty.$}
\end{equation}
If $\,G\in\mathcal{LP},$ then $\,G_\alpha\in\mathcal{LP}\,$ for each $\,\alpha\ge 0,$ where
$G_\alpha$ is defined by
$$G_\alpha(z) = \left(z\frac{d}{dz} +\alpha\right) G(z)\qquad(z\in\mathbb{C}).$$
\end{theorem}

\begin{proof} Since the case $\,\alpha =0\,$ is known, we shall only prove the
case $\,\alpha>0\,$ although the proof for the case $\,\alpha =0\,$ is not much different.

The hypothesis \eqref{condLP} implies that $G(z)$ has order not exceeding one.
Let $\big(\zeta_m\big)$ denote the sequence of all distinct positive zeros of $G(z)$,
arranged in ascending order of magnitude,
and $\ell_m$ the multiplicity of $\zeta_m$. Since $G(z)$ does not vanish at the origin, the order
restriction implies
\begin{equation}\label{LPC1}
G(z) = \gamma_0 \prod_{m=1}^\infty \left(1-\frac{z^2}{\zeta_m^2}\right)^{\ell_m},
\quad\text{where}\quad \sum_{m=1}^\infty \frac{\ell_m}{\,\zeta_m^2\,}<\infty.
\end{equation}

As readily verified, $G_\alpha(z)$ is an even entire function with
$$G_\alpha(z) = \sum_{m=0}^\infty \frac{\,(-1)^m\gamma_m\,(2m+\alpha)\,}{m!}\,z^{2m}.$$
We note that $G_\alpha(z)$ also has order not exceeding one because the term $\,\log |2m+\alpha|\,$
does not give any contribution in calculating order.
Concerning the zeros of $G_\alpha(z)$, we first observe that $\,G_\alpha(\zeta_m) =0\,$
only when $\,\ell_m\ge 2.$ Since the $j$th derivative of $G_\alpha(z)$ is given by
$$G_\alpha^{(j)}(z) = z G^{(j+1)}(z) + (\alpha+j) G^{(j)}(z),\quad j=1,2,\cdots,$$
it is evident that $\zeta_m$ has multiplicity $\ell_m -1$ in such a case.

By taking the logarithmic derivative of $z^{\alpha}G(z)$
with the representation of \eqref{LPC1} and simplifying, we obtain
\begin{equation}\label{LPC2}
    \frac{\,G_\alpha(z)\,}{z G(z)}
    =\frac{\alpha}{z}+ \sum_{m=1}^\infty\ell_m \left(\frac{1}{z-\zeta_m}+\frac{1}{z+\zeta_m}\right).
\end{equation}
Since the coefficients are all positive, according to Remark \ref{remarkHP} after
the proof of Theorem \ref{lemmaP}, this partial fraction expansion indicates that
$G_\alpha(z)$ has only an infinity of real zeros and each of the intervals
\begin{equation*}
\big(0, \,\zeta_1\big],\,\,\big(\zeta_m,\,\zeta_{m+1}\big],\quad m=1,2,\cdots,
\end{equation*}
contains at most two zeros. In addition, \eqref{LPC2} gives
\begin{equation*}
W\big[G_\alpha, G\big](x) = 4x [G(x)]^2\,\sum_{m=1}^\infty \frac{\ell_m \zeta_m^2}{\,(x^2 -\zeta_m^2)^2\,}>0
\end{equation*}
for all $\,x>0,\,x\ne \zeta_m,\,m=1,2,\cdots,$ and thus
each positive zero of $G_\alpha(z)$ is simple if it does not
coincides with some $\zeta_m$.

Let $\big(\sigma_m\big)$ denote the
sequence of all positive zeros of $G_\alpha(z)$, repeated
up to multiplicities and arranged in the manner
$\,0<\sigma_1\le\sigma_2\le\cdots.$
Due to its nature explained as above, it is evident that
\begin{equation}\label{LPC3}
\sum_{m=1}^\infty \frac{1}{\,\sigma_m^2\,}
\le \frac{1}{\,\sigma_1^2} + \sum_{m=1}^\infty \frac{\ell_m}{\,\zeta_{m}^2\,}<\infty.
\end{equation}
By applying the Hadamard theorem, we conclude that $\,G_\alpha\in\mathcal{LP}\,$ with
\begin{equation*}
G_\alpha(z) = \alpha \gamma_0 \prod_{m=1}^\infty \left(1-\frac{z^2}{\sigma_m^2}\right).
\end{equation*}
\end{proof}

\begin{remark}\label{remarkLP}
In dealing with even entire functions of $\mathcal{LP}$,
it is customary and often advantageous to consider the subclass $\,\mathcal{LP}^+\subset \mathcal{LP},$
which consists of all real entire functions $g(z)$ representable in the form
\begin{equation}\label{LP2}
g(z) = A z^\ell e^{\beta  z} \prod_{m=1}^\omega \left(1+\frac{z}{\tau_m}
\right),\quad 0\le\omega\le \infty,
\end{equation}
where $\,\beta\ge 0,\,A\in\R,$ $\ell$ is a nonnegative integer, and $\big(\tau_m\big)$
is a sequence of positive reals satisfying $\,\sum_{m=1}^\omega 1/\tau_m<\infty.$

As readily observed, $\,g\in \mathcal{LP}^+\,$ if and only if
$\,G\in \mathcal{LP},$ where $G$ is an even extension of $g$ defined by
$\,G(z) = g(-\delta z^2)\,$ for some $\,\delta >0.$ As a consequence,
Theorem \ref{lemmaLP} may be rephrased as follows: If $\,g\in\mathcal{LP}^+\,$ and takes the form
\begin{equation}\label{typeLP2}
g(z) = \sum_{m=0}^\infty \frac{\,\gamma_m\,}{m!}\,z^{m}\qquad(z\in\mathbb{C}),
\end{equation}
where $(\gamma_m)$ is a sequence of non-zero reals satisfying \eqref{condLP},
then
$$\left(z \frac{d}{dz} +\alpha\right) g\in\mathcal{LP}^+ \quad\text{for all}
\quad \alpha\ge 0.$$
\end{remark}

\begin{corollary} \label{corollaryLP}
Under the same assumptions of Theorem \ref{theoremLP}, we have
$$\left(z \frac{d}{dz} +\alpha\right) \bH_\nu(f) \in\mathcal{LP} \quad\text{for all}
\quad \alpha\ge 0.$$
\end{corollary}

\section{Sturm's method for the sign pattern}
This section focuses on investigating whether the sequence
\begin{align}\label{S0}
\bH_\nu(f)\rb{j_{\mu, m}}&= \int_0^1 t^{\nu+1/2} f(t) \bJ_\nu\rb{j_{\mu, m} t} dt\nonumber\\
&= \frac{1}{j_{\mu, m}^{\nu+3/2}}\int_0^{j_{\mu, m}} f\left(\frac{t}{j_{\mu, m}}\right)
\left[t^{\nu+1/2}\bJ_\nu(t)\right] dt
\end{align}
alternates in sign in the particular case $\mu=\nu$. Our investigation will be based on
a version of Sturm's comparison theorems which states as follows.
(It is practically due to Sturm \cite{S} but we refer to \cite{LM}, \cite{Ma1}, \cite{Wa} for the
present version and further backgrounds with applications.)

\begin{itemize}{\it
\item[{}] For $\,\phi_1, \phi_2\in C([a, b]),$ let $u_1(t), u_2(t)$ be $C^1$ solutions of
$$ u_1'' + \phi_1(t) u_1 =0, \quad u_2'' + \phi_2(t) u_2 =0\quad(a<t<b),$$
respectively, such that $\,u_1(a) = u_2(a) =0\,$ and $\,u_1'(a+) = u_2'(a+) >0.$
If $\,\phi_2(t)>\phi_1(t)\,$ for all $\,a<t<b,$ then $\,u_1(t)>u_2(t)\,$ for all $t$ between
$a$ and the first zero of $u_2(t)$, and hence the first zero of $u_2(t)$ on $(a, b]$ is on the
left of the first zero of $u_1(t)$.}
\end{itemize}

In what follows we shall set $\,j_{\nu, m} \equiv j_m\,$ for each $m$ to simplify notation.

\subsection{The case $\,|\nu|<1/2$}
As readily verified, the function $\,u(t) = t^{\nu+1/2}\bJ_\nu(t),\,t>0,$ satisfies the same differential equation
as defined in \eqref{SL1}, that is,
\begin{equation*}\label{S1}
u'' + \phi(t) u =0,\quad\text{where}\quad \phi(t) = 1+ \frac{(1/4-\nu^2)}{t^2}.
\end{equation*}
In addition, we note that $u(t)$ is strictly positive on $\,\big(0, j_1\big),\,\big(j_{2k},\,j_{2k+1}\big)\,$
and strictly negative on $\,\big(j_{2k-1},\,j_{2k}\big)\,$ for each $\,k=1,2,\cdots.$

For a fixed $k$, if we consider the functions $u_1(t), u_2(t)$ defined by
$$u_1(t) = u\left(j_{2k} +t\right),\quad u_2(t) = -u\left(j_{2k} -t\right),$$
it is readily verified that $u_1(t), u_2(t)$ satisfy
\begin{align*}
u_1'' + \phi\left(j_{2k} +t\right)& u_1 =0, \quad u_2'' + \phi\left(j_{2k} -t\right)u_2 =0,\\
u_1(0) &= u_2(0) =0,\\
u_1'(0) = u_2'(0) &= -\frac{j_{2k}^{\nu+3/2}}{2(\nu+1)} \bJ_{\nu+1}
\rb{j_{2k}}>0.
\end{align*}
Since $\phi(t)$ is strictly decreasing for $\,t>0\,$ when $\,|\nu|<1/2,$
it follows from the aforementioned Sturm's theorem that
$\,j_{2k+1}-j_{2k}>j_{2k} - j_{2k-1}\,$ and $\,u_1(t)> u_2(t)\,$ for $\,0<t< j_{2k}- j_{2k-1},$
which implies in turn that
$$g\rb{j_{2k} +t} u_1(t) >g\rb{j_{2k}-t} u_2(t),\quad 0<t< j_{2k}- j_{2k-1},$$
for any positive function $g(t)$ increasing on the interval $\,\big[j_{2k-1},\,
j_{2k+1}\big].$

As a consequence, we find by integrating that
\begin{align*}
\int_0^{j_{2k+1}-j_{2k}}g(j_{2k}+t)u_1(t)dt &=\left(\int_0^{j_{2k}-j_{2k-1}} +\int_{j_{2k}-j_{2k-1}}^{j_{2k+1}-j_{2k}}\right)(\cdots)\,dt\\
&> \int_0^{j_{2k}-j_{2k-1}}g(j_{2k}-t)u_2(t)dt.
\end{align*}
On changing variables appropriately, the last inequality leads to
\begin{equation*}
\int_{j_{2k}}^{j_{2k+1}}g(t)u(t)dt> -\int_{j_{2k-1}}^{j_{2k}}g(t)u(t)dt,
\end{equation*}
which implies by the additivity of integrals the inequality
\begin{equation}\label{S2}
    \int_{j_{2k-1}}^{j_{2k+1}}g(t)u(t) dt>0.
\end{equation}

In the same manner, if we replace $u_1(t), u_2(t)$ by the functions
$$\widetilde{u_1}(t) = -u\left(j_{2k+1} +t\right),\quad \widetilde{u_2}(t) = u\left(j_{2k+1} -t\right)$$
and proceeds as above, it is not difficult to deduce the inequality
\begin{equation}\label{S3}
    \int_{j_{2k}}^{j_{2k+2}}g(t)u(t) dt<0.
\end{equation}
By setting $\,j_0 =0,$ this inequality continues to be valid for $\,k=0.$

Let $f(t)$ be positive, increasing and integrable for $\,0<t<1.$
By applying \eqref{S2} with $\,g(t) = f(t/j_{2m+1}),$ we find that
\begin{align*}
\int_0^{j_{2m+1}} f\left(\frac{t}{j_{2m+1}}\right)
\left[t^{\nu+1/2}\bJ_\nu(t)\right] dt =\left(\int_0^{j_1} +\sum_{k=1}^m\int_{j_{2k-1}}^{j_{2k+1}} \right)
(\cdots) \,dt>0.
\end{align*}
Similarly, by applying \eqref{S3} with $\,g(t) = f(t/j_{2m}),$ we find that
\begin{align*}
\int_0^{j_{2m}} f\left(\frac{t}{j_{2m}}\right)
\left[t^{\nu+1/2}\bJ_\nu(t)\right] dt =\sum_{k=1}^m\int_{j_{2k-2}}^{j_{2k}}
(\cdots) \,dt<0.
\end{align*}

In view of \eqref{S0}, hence, we have proved the following sign pattern.

\begin{lemma}\label{lemmaS1}
Let $\,|\nu|<1/2.$ If $f(t)$ is positive, increasing and integrable for $\,0<t<1,$ then
\begin{equation*}
    {\rm{sgn}}\,\big[\bH_\nu(f)\left(j_{\nu,m}\right)\big] =(-1)^{m+1},\quad m=1, 2, \cdots.
\end{equation*}
\end{lemma}

\begin{remark}\label{remarkPZ}
When $\,|\nu| = 1/2$, it is shown by P\'olya \cite{P} that if $f(t)$ is positive, increasing and integrable, then
the same sign pattern holds true unless $f$ belongs to the so-called exceptional case,
the class of all step functions on $[0, 1]$ having finitely many jump discontinuities
at rational points.
\end{remark}

\subsection{The case $\,|\nu|>1/2$}
We shall modify Makai \cite{Ma2} by considering the function
\begin{equation*}
    w(t) = t^{\frac 12 \left( 1+ \frac{\nu}{|\nu|}\right) }\,\bJ_\nu\rb{t^{ \frac{1}{2|\nu|}}} \qquad(t>0),
\end{equation*}
which is easily seen to be a solution of the differential equation
\begin{equation}\label{S4}
w'' + \varphi(t) w =0,\quad\text{where}\quad \varphi(t) = \frac{1}{4 \nu^2}\,t^{\frac{1}{|\nu|} -2} .
\end{equation}
We note that $\varphi(t)$ is strictly decreasing for $\,t>0\,$ in the present case.

Let us use the notation $\,\hj_m := j_{m}^{2|\nu|}\,$ so that $w(t)$ has zeros at
each $\hj_m$. As in the previous case, if we fix a positive integer $k$ and consider
$$w_1(t) = -w\left(\hj_{ 2k+1} +t\right),\quad w_2(t) = w\left(\hj_{ 2k+1} -t\right),$$
then it is straightforward to calculate, by using \eqref{S4}, that
\begin{align*}
w_1'' + \varphi\left(\hj_{ 2k+1} +t\right)& w_1 =0, \quad w_2'' + \varphi\left(\hj_{ 2k+1} -t\right)w_2 =0,\\
w_1(0) &= w_2(0) =0,\\
w_1'(0) = w_2'(0) &= -\frac{j_{ 2k+1}^{2+\nu-|\nu|}}{\,4|\nu(\nu+1)|\,} \,\bJ_{\nu+1}
\rb{j_{2k+1}}>0,
\end{align*}
which implies by Sturm's theorem that
$\,\hj_{2k+2}-\hj_{2k+1}>\hj_{2k+1} - \hj_{2k}\,$ and also $\,w_1(t)> w_2(t)\,$ for $\,0<t< \hj_{2k+1}- \hj_{2k}.$

Proceeding as before, we find that
\begin{equation}\label{S5}
    \int_{\hj_{2k}}^{\hj_{2k+2}}g(t)w(t) dt<0,\quad k=1, 2, \cdots,
\end{equation}
for any positive function $g(t)$ increasing on $\,\big[\hj_{2k},\,
\hj_{2k+2}\big].$ Set $\,\hj_0 =0$. In the case $\nu>1/2$, it is easy to see that \eqref{S5}
remains valid for $\,k=0.$  In the case $\,-1<\nu <-1/2$, we use the fact
$\,w_1(t)>w_2(t)\,$ for $0<t<\hj_1$ to observe
\begin{equation*}
    \int_{0}^{2\hj_1}g(t)w(t) dt <0.
\end{equation*}
Since $\,w(t)<0\,$ for $\,2\hj_1<t<\hj_2,$ it is now evident that
\begin{equation*}
    \int_{0}^{\hj_2}g(t)w(t) dt <0,
\end{equation*}
whence  \eqref{S5} remains valid for $\,k=0.$ In a like manner, we have
\begin{equation}\label{S6}
    \int_{\hj_{2k-1}}^{\hj_{2k+1}}g(t)w(t) dt>0, \quad k=1, 2, \cdots.
\end{equation}

If $f(t)$ is a positive function for $\,0<t<1\,$ and $\,x>0,$ then
\begin{align*}
&\int_0^{x} f\left(\frac{t}{x}\right)
\left[t^{\nu+1/2}\bJ_\nu(t)\right] dt\\
&\qquad = \frac{1}{2|\nu|} \int_0^{x^{2|\nu|}} t^{(3/2-3|\nu|)/2|\nu|} f\left(\frac{t^{1/2|\nu|}}{x}\right) \,w(t) dt
\end{align*}
whenever the integral converges. We note that the function
$$h(t, x) = t^{(3/2-3|\nu|)/2|\nu|} f\left(\frac{t^{1/2|\nu|}}{x}\right),\quad 0<t<x^{2|\nu|},$$
is increasing in $t$ if and only if $\,t^{3/2-3|\nu|} f(t)\,$ is increasing for $\,0<t<1.$
By applying \eqref{S5}, together with the additional case $\,k=0,$ and \eqref{S6} with
$\,g(t) = h\rb{t, j_{2m+1}}\,$ or $\,h\rb{t, j_{2m}},$ we find that
\begin{align*}
&\int_0^{j_{2m+1}} f\left(\frac{t}{j_{2m+1}}\right)
\left[t^{\nu+1/2}\bJ_\nu(t)\right] dt >0,\\
&\quad \int_0^{j_{2m}} f\left(\frac{t}{j_{2m}}\right)
\left[t^{\nu+1/2}\bJ_\nu(t)\right] dt <0, \quad m=1, 2, \cdots,
\end{align*}
provided that $\,t^{3/2-3|\nu|} f(t)\,$ is positive and increasing
for $\,0<t<1.$

In summary, it is proved that the following sign pattern holds true.

\begin{lemma}\label{lemmaS2}
Let $\,|\nu|>1/2\,$ and $f(t)$ be a positive integrable function satisfying \eqref{IA} such that $\,t^{3/2-3|\nu|} f(t)\,$ is increasing
for $\,0<t<1,$ then
\begin{equation*}
    {\rm{sgn}}\,\big[\bH_\nu(f)\left(j_{\nu,m}\right)\big] =(-1)^{m+1},\quad m=1, 2, \cdots.
\end{equation*}
\end{lemma}

\begin{remark}
As it may be expected, the sufficient conditions for the sign changes presented in
both Lemmas \ref{lemmaS1}, \ref{lemmaS2} are far from being optimal.
For example, if we take $\,f(t)= 2(\nu+1) \,t^{\nu+1/2},$ a special case of  \eqref{EX1},
then its Hankel transform
$\,\bH_\nu(f)(z) = \bJ_{\nu+1}(z)\,$ is subject to the sign pattern
\begin{equation*}
    {\rm{sgn}}\,\big[\bJ_{\nu+1} \left(j_{\nu,m}\right)\big] =(-1)^{m+1}\quad
    \text{for all $m$ and $\,\nu>-1.$}
\end{equation*}
Nevertheless, both lemmas indicate that this sign pattern holds true only in the range
$\,-1<\nu\le 1,\,\nu \ne -1/2.$
\end{remark}

On combining the above lemmas with Theorems \ref{theoremZ1} and \ref{theoremLP},
we obtain the following theorem which constitutes one of our main results.

\begin{theorem}\label{theoremZ2} For $\,\nu>-1,$ suppose that $f(t)$ is a positive integrable function
 satisfying \eqref{IA} and the following case assumptions:
 \begin{itemize}
 \item[\rm(i)] $f(t)$ is increasing for $\,0<t<1\,$ when $\,|\nu|\le 1/2\,$  and it does not belong to the exceptional case
when $\,|\nu| =1/2.$
 \item[\rm(ii)] The function $\,t^{3/2-3|\nu|} f(t)\,$ is increasing for $\,0<t<1\,$ when $\,|\nu|>1/2.$
 \end{itemize}
 Then $\,\bH_\nu(f)\in\mathcal{LP}\,$ with an infinity of real simple zeros. Moreover, $\bH_\nu(f)(z)$
 has one and only one positive zero in each of the intervals
 $$\,\big(j_{\nu, m}, \, j_{\nu, m+1}\big),\quad m=1, 2, \cdots,$$
 and no positive zeros elsewhere.
\end{theorem}

While we shall present applications in the next section,
we refer to \cite{P} for improved bounds of zeros when $\,|\nu| =1/2\,$ and \cite{E}, \cite{Wa} for
a number of tabulated Hankel transforms for which
Theorem \ref{theoremZ2} may be applicable.

\section{${}_1F_2$ hypergeometric functions}
In application of the results established in this paper, this section
aims to investigate the zeros of ${}_1F_2$ hypergeometric functions of the form
\begin{equation}\label{HG1}
\Phi(z) = {}_1F_2 \left[\begin{array}{c}
a\\ b, \,c\end{array}\biggr| - \frac{\,z^2}{4} \right]\qquad(z\in\mathbb{C}),
\end{equation}
where $a, b, c$ are real numbers subject to the condition
neither of $b, c$ coincides with a non-positive integer.  We note that $\Phi(z)$
is an even real entire function of order one with $\,\Phi(0)=1.$
Of our main concern is to determine the set of parameters for which
$\Phi(z)$ belongs to the Laguerre-P\'olya class $\mathcal{LP}.$

In consideration of the corresponding entire function
\begin{equation}\label{HG2}
\phi(z) = {}_1F_2 \left[\begin{array}{c}
a\\ b, \,c\end{array}\biggr|  z\right]\qquad(z\in\mathbb{C}),
\end{equation}
related by $\,\Phi(z) =\phi(-z^2/4),$ it should be emphasized that \emph{any condition on parameters $a, b, c$ for
$\,\Phi\in \mathcal{LP}\,$ is also sufficient for $\,\phi\in\mathcal{LP}^+$, without altering any material, }
due to the relationship between $\mathcal{LP}$ and $\mathcal{LP}^+$ as described in Remark \ref{remarkLP}.
Therefore all of our subsequent results are relevant to the open problem raised by Sokal \cite{So}
concerning $\,\phi\in\mathcal{LP}^+$.

\subsection{Positivity and zeros}
As to the existence of zeros, the following results have been established in our recent work
\cite[Theorem 4.2]{CCY} on the basis of Gasper's sums of squares method \cite{Ga} and the known asymptotic
behavior of $\Phi(z)$.

\begin{proposition}\label{propositionP1}
For each $\,a>0,$ define
\begin{align*}
\mathcal{N}_a &= \left\{(b, c): \text{$\,b\le a \,$ or $\,c\le a\,$ or $\,b+c<3a + \frac 12\,$} \right\},\\
\mathcal{P}_a &= \left\{(b, c) : b>a,\,\,c\ge\max\Big[ 3a+ \frac 12 -b,\,\, a+\frac{a}{2(b-a)}\Big]\right\},\\
\mathcal{S}_a  &= \left\{\big(a+1/2,\,2a\big),\,\,\big(2a,\, a+1/2\big)\right\}.
\end{align*}

\begin{itemize}
\item[\rm(i)] If $\,(b, c)\in\mathcal{N}_a,$ then $\Phi(z)$ has at least one positive zero.
\item[\rm(ii)] If $\,(b, c)\in \mathcal{P}_a\setminus\mathcal{S}_a,$ then $\,\Phi(x)>0\,$ for all $\,x>0\,$ and hence $\Phi(z)$ has no real zeros.
If $\,(b, c)\in \mathcal{S}_a,$ then $\Phi(z)$ reduces to
\begin{align*}
{}_1F_2 \left[\begin{array}{c}
a\\ a+1/2, \,2a\end{array}\biggr| - \frac{\,z^2}{4} \right] = \bJ_{a-1/2}^2\left(\frac z2\right),
\end{align*}
which has only real and double zeros at $\,2j_{a-1/2, k},\,k=1,2,\cdots.$
\end{itemize}
\end{proposition}

As shown in Figures \ref{Fig1}, \ref{Fig2} for some special values of $a$,
$\mathcal{P}_a$ represents an infinite hyperbolic region in $\R_+^2$
containing the so-called Newton diagram or polyhedron of $\mathcal{S}_a$
(see \cite{CC2} for the definition and related results).

Proposition \ref{propositionP1} gives rise to the decomposition
$$\mathbb{R}_+^2 = \mathcal{P}_a \cup \mathcal{N}_a\cup \left(\mathcal{P}_a\cup \mathcal{N}_a\right)^c.$$
Regarding the region $\mathcal{P}_a\,$ of parameters, it is simple to observe
\begin{theorem}\label{test1}
For each $\,a>0,$ if $\,(b, c)\in\mathcal{P}_a\setminus\mathcal{S}_a,$ then
$\,\Phi\notin \mathcal{LP}\,$ and $\Phi(z)$ has an infinity of complex zeros all of which are not purely
imaginary.
\end{theorem}

\begin{proof}
As readily calculated, the entire function $\phi(z)$, defined by \eqref{HG2}, has growth order $1/2$ .
Consequently, it has an infinity of zeros under the present setting (see \cite[p. 31]{Le})
and so does the function $\Phi(z)$. Moreover,
$$\Phi(iy) = \sum_{k=0}^\infty \frac{(a)_k}{\,k! (b)_k (c)_k\,}\left(\frac{y^2}{4}\right)^k>0
\quad\text{for all $\,y\in\R\,$}$$
and hence $\Phi(z)$ can not have purely imaginary zeros. For $\,(b, c)\in\mathcal{P}_a\setminus\mathcal{S}_a,$
since $\Phi(z)$ has no real zeros, the conclusion follows.
\end{proof}

\subsection{Images of Hankel transforms}

What matters is the nature of zeros in the case $\,(b, c)\in\mathcal{N}_a.$
If $\, b=a\,$ or $\,c=a,$ then $\Phi(z)$ is equal to $\bJ_{c-1}(z)$ or $\bJ_{b-1}(z)$, respectively.
If one of $b, c$ exceeds $a$, it is possible to
recognize $\Phi(z)$ as a Hankel transform of type \eqref{H2}. By considering symmetry of $\Phi(z)$
with respect to parameters $b, c$, let us assume  $\,b>a.$ By integrating termwise, it is readily verified that
\begin{align*}
\Phi(z) &= \frac{2}{B(a, b-a)} \,\int_0^1 (1-t^2)^{b-a-1} t^{2a-1}\bJ_{c-1}(zt) dt\\
&= \frac{2}{B(a, b-a)}\bH_{c-1}(f)(z),\quad f(t) =(1-t^2)^{b-a-1} t^{2a-c-1/2}.
\end{align*}

Setting $\,g(t) =  t^{3/2-3\nu} f(t),\, h(t) =  t^{3/2+3\nu} f(t),$ with $\,\nu=c-1,$ that is,
\begin{align*}
g(t) = (1-t^2)^{b-a-1} t^{2a-4c +4},\quad h(t) = (1-t^2)^{b-a-1} t^{2a+2c-2},
\end{align*}
it is elementary to observe the following aspects:
\begin{itemize}
\item $f(t)$ is integrable and satisfies \eqref{IA} when $\,c<2a+1/2.$
\item $f(t)$ is increasing for $\,0<t<1\,$ when $\,b\le a+1,\,2a-c-1/2\ge 0\,$;
$g(t)$ is increasing for $\,0<t<1\,$ when $\, b\le a+1,\,a-2c+2\ge 0\,$;
$h(t)$ is increasing for $\,0<t<1\,$ when $\, b\le a+1,\,a+c-1\ge 0.$
\item $f(t)$ is in the exceptional case only when $\,b=a+1,\,c=2a-1/2.$
\end{itemize}

By applying Theorem \ref{theoremZ2} and making use of symmetry, it is a matter of arranging parameters
to obtain the following information, where we exclude the trivial cases $\,b=a\,$ or $\,c=a.$ As before,
all positive zeros of $\Phi(z)$ will be denoted by $\big(\zeta_k\big)$ arranged in ascending order of magnitude.

\begin{theorem}
For each $\,a>1/2,$  let $\mathcal{Z}_a$ denote the set of all ordered-parameter pairs
$\,(b, c)\in\R_+^2\,$ defined by
$\,\mathcal{Z}_a = \big(a, \,a+1\big]\times I_a,$ where
\begin{align*}
I_a =\left\{\begin{aligned}
&{\Big[1-a,\,2a- \frac 12 \Big] \quad\text{if $\,\frac 12 <a<1,$}}\\
&{\Big(0,\,\frac 12 a  +1\Big]\quad\text{if $\,a\ge 1,$}}\end{aligned}\right.
\end{align*}
and $\,\mathcal{Z}_a^* = \left\{ (b, c) :  (c, b)\in \mathcal{Z}_a\right\}.$
If $\,(b, c)\in \big(\mathcal{Z}_a \cup  \mathcal{Z}_a^*\big),$ then $\,\Phi\in \mathcal{LP}\,$
and the positive zeros of $\Phi(z)$ are all simple and satisfy the following:
\begin{align*}
&{\rm{(i)}}\quad\left\{\begin{aligned}
&{j_{c-1, k}<\zeta_k<j_{c-1, k+1}\quad\text{for}\quad (b, c)\in \mathcal{Z}_a,}\\
&{j_{b-1, k}<\zeta_k<j_{b-1, k+1}\quad\text{for}\quad (b, c)\in \mathcal{Z}_a^*,}
\end{aligned}\right.\quad k=1, 2, \cdots.\\
&{\rm{(ii)} }\quad
\sum_{k=1}^\infty \frac{1}{\,\zeta_k^2\,} = \frac{a}{4bc}.
\end{align*}
\end{theorem}

The last explicit sum in (ii) results from formula \eqref{F4} due to
 \begin{align*}
 \beta_k(f) = \int_0^1 t^{2k+c-1/2}f(t) \,dt = \frac{\,(a)_k\,}{(b)_k},\quad k=0, 1, 2,\cdots.
 \end{align*}

\subsection{A transference principle}
In summary, it has been shown so far that $\,\Phi\in\mathcal{LP}\,$ in the case when $\,a>0\,$ and
parameters $b,c$ satisfy one of the following conditions:
\begin{equation*}\label{set1}
(1)\,\text{$\,b=a\,$ or $\,c=a.$}\quad (2)\,\, (b, c)\in\mathcal{S}_a.\quad
(3)\,\,(b, c)\in \big(\mathcal{Z}_a \cup  \mathcal{Z}_a^*\big).
\end{equation*}

As it may be expected from Theorem \ref{lemmaLP}, which asserts that
certain class of $\mathcal{LP}$ is closed under the operation $\,z \cdot d/dz +\alpha,\,\alpha\ge 0,$
it is possible to extend the range of parameters with the aid of the following.

\begin{lemma}\label{lemmaT}
If $\Phi(z)$ belongs to $\mathcal{LP},$ then the function of type
\begin{equation}\label{PS1}
{}_1F_2 \left[\begin{array}{c}
a+m\\ b+n, \,c+\ell\end{array}\biggr| - \frac{\,z^2}{4} \right]
\end{equation}
also belongs to $\mathcal{LP},$ where $m, n, \ell$ are integers subject to the condition
$$ m\ge 0,\quad -b<n\le m,\quad -c<\ell\le m.$$
\end{lemma}

\begin{proof}
As an even entire function of order one,  it is easy to verify that
any entire function of the above type falls under the scope of
Theorem \ref{lemmaLP}. For $\,a\ne 0,\, b\ne 1,\,c\ne 1,$ if we consider the differential operators
\begin{equation}\label{oper}
D = z \frac{d}{dz},\quad D_b = z \frac{d}{dz} + 2(b-1), \quad D_c = z \frac{d}{dz} + 2(c-1),
\end{equation}
then it is routine to calculate
\begin{align*}
D\big[\Phi(z)\big] &= - \frac{\,a z^2\,}{2bc}\, {}_1F_2 \left[\begin{array}{c}
a+1\\ b+1, \,c+1\end{array}\biggr| - \frac{\,z^2}{4} \right],\\
D_b\big[\Phi(z)\big] &=  2(b-1)\,{}_1F_2 \left[\begin{array}{c}
a\\ b-1, \,c\end{array}\biggr| - \frac{\,z^2}{4} \right],\\
D_c\big[\Phi(z)\big] &= 2(c-1)\,{}_1F_2 \left[\begin{array}{c}
a\\ b, \,c-1\end{array}\biggr| - \frac{\,z^2}{4} \right].
\end{align*}

Since \eqref{PS1} can be obtained from $\Phi(z)$ by a successive applications of the operations \eqref{oper},
where we eliminate the multiplicative factor in each application,
the desired result follows by Theorem \ref{lemmaLP}.
\end{proof}

\subsection{Extensions of parameters}
{\bf 1.} In the first place, if we consider the case $\, b=a,$ then $\,\Phi(z) = \bJ_{c-1}(z),$
a member of $\mathcal{LP}$ for any $\,c>0,$ and hence the above transference principle
gives the following result known to Sokal \cite[Theorem 5]{So}:

\begin{proposition}\label{propositionP2}
If $a$ does not coincide with a non-positive integer and $\,c>0,$ then
the function of type
\begin{equation}
{}_1F_2 \left[\begin{array}{c}
a+m\\ a+n, \,c\end{array}\biggr| - \frac{\,z^2}{4} \right]
\end{equation}
belongs to $\mathcal{LP}$,  where $m, n$ are integers subject to the condition
$$\,m\ge 0,\quad -a<n\le m.$$
\end{proposition}

\medskip
{\bf 2.} We next consider extending the case $\,a>0,\,(b, c)\in\mathcal{S}_a.$
Since $\mathcal{LP}$ is closed under product and scaling of arguments,
the ${}_2F_3$ hypergeometric function obtained by (\cite[\S 5.41, (1)]{Wa}, \cite[\S 6.2, (39)]{Lu})
\begin{align}\label{ProB}
\bJ_\mu(z)\bJ_\nu(z) =
{}_2F_3 \left[\begin{array}{c} (\mu+\nu+1)/2,\, (\mu+\nu+2)/2\\
\mu+1,\,\nu+1,\,\mu+\nu+1\end{array}\biggr| - z^2\right]
\end{align}
belongs to $\mathcal{LP}$ for any $\,\mu>-1,\,\nu>-1\,$ unless $\,\mu+\nu+1 =0.$ In the case
$\,\mu=\nu,$ it reduces to the square of \emph{normalized} Bessel function
$$\bJ_\nu^2(z)=
{}_1F_2 \left[\begin{array}{c} \nu+1/2\\
\nu+1,\, 2\nu+1\end{array}\biggr| - z^2 \right],\quad \nu\ne -1/2,$$
and the transference principle gives rise to the following result.

\begin{proposition}\label{propositionP3}
For $\, a>-1/2,\,a\ne 0,$  the function of type
\begin{equation}
{}_1F_2 \left[\begin{array}{c}
a+m\\ a+1/2+n, \,2a+\ell\end{array}\biggr| - \frac{\,z^2}{4} \right]
\end{equation}
belongs to $\mathcal{LP}$, where $m, n, \ell$ are integers subject to the condition
$$\,m\ge 0,\quad -a-1/2<n\le m, \quad -2a<\ell\le m.$$
\end{proposition}

This result in the trivial instance $\,m=n=\ell=0\,$ was observed by several authors
(see \cite{So} for the relevant references). Another special case of interest arises when
$\,\mu=-\nu\,$ for which  \eqref{ProB} reduces to
$$\bJ_{-\nu}(z)\bJ_\nu(z)=
{}_1F_2 \left[\begin{array}{c} 1/2\\
-\nu+1,\, \nu+1\end{array}\biggr| - z^2 \right],\quad -1< \nu<1.$$

Since $\nu$ is free to vary, it is not difficult to see that the transference principle leads to the following
family of $\mathcal{LP}$ functions.

\begin{theorem}\label{propositionP4}
Given an integer $\, m\ge 0,$ the function of type
\begin{equation}\label{CS}
{}_1F_2 \left[\begin{array}{c}
1/2+m\\ b, \,c\end{array}\biggr| - \frac{\,z^2}{4} \right]
\end{equation}
belongs to $\mathcal{LP}$ for all $\,(b, c)\in\R_+^2\,$ lying on the line segments
$$b+c =k,\quad k = 1, 2, \cdots, 2m+2.$$
\end{theorem}

\begin{figure}[!ht]
 \centering
 \includegraphics[width=280pt, height= 280pt]{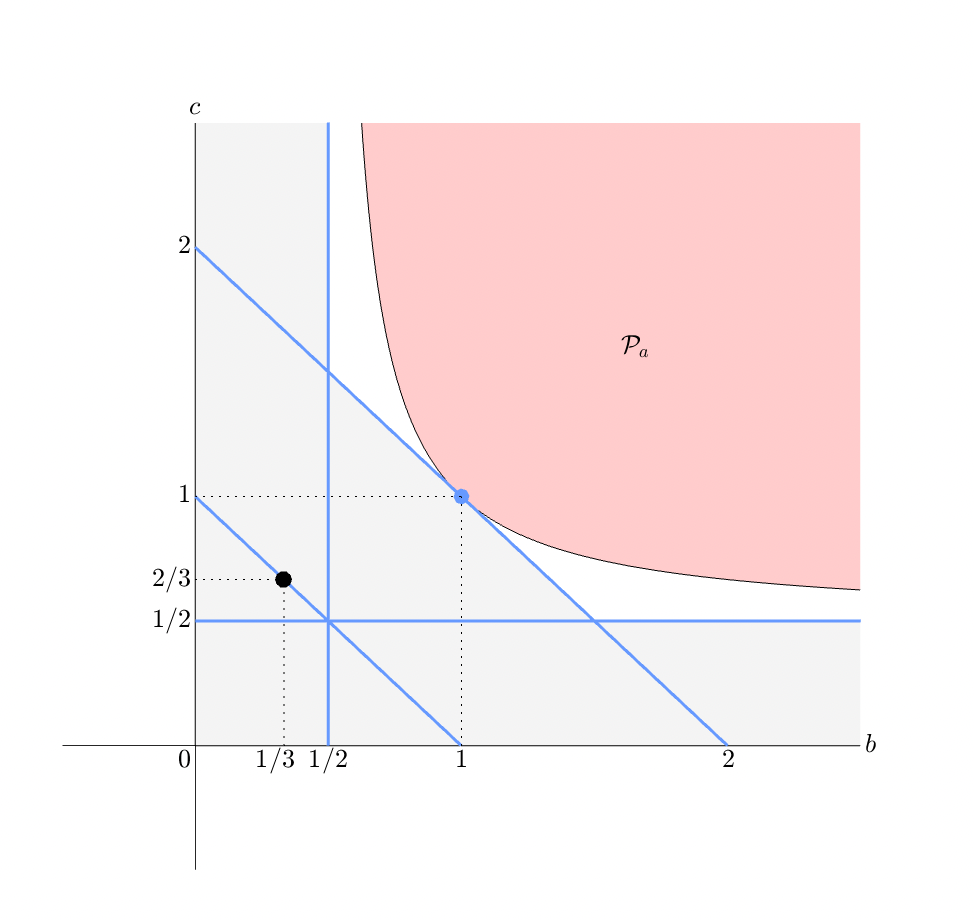}
 \caption{For $\,a=1/2,$ the region $\Lambda$ of parameter-pair $(b, c)$ for $\,\Phi\in\mathcal{LP}\,$
 consists of line segments, vertical and horizontal rays. If
 $\,(b, c)\in \mathcal{P}_a\setminus\{(1, 1)\},$ then $\,\Phi\notin \mathcal{LP}\,$ with infinitely many complex zeros.}
\label{Fig1}
\end{figure}

This non-trivial result appears to be unavailable in the literature. For example, if we take $\,a=1/2\,$ in
the definition of \eqref{HG1} and put 
\begin{align}\label{region1}
\Lambda &= \big\{(b, c)\in\R_+^2 : b+c =1\,\,\text{or}\,\,b+c =2\big\}\nonumber\\
&\qquad\cup\big\{b=1/2,\,c>0\big\}\,\,\cup\,\,\big\{b>0,\,c=1/2\big\},\nonumber\\
\mathcal{P}_a &= \big\{(b, c)\in\R_+^2 : b+c\ge 2,\,\,(2b-1)c\ge b,\,\,b>1/2\big\},
\end{align}
then Proposition \ref{propositionP2} and Theorem \ref{propositionP4} indicate that $\,\Phi\in\mathcal{LP}\,$ for $\,(b, c)\in \Lambda\,$ and
Theorem \ref{test1} shows that $\,\Phi\notin \mathcal{LP}\,$ for $\,(b, c)\in \mathcal{P}_a,\,(b, c)\ne (1, 1),$ with an infinity of complex zeros
(see Figure \ref{Fig1}).

We remark that $\Lambda$ contains the point $(1/3, 2/3),$ lying on the line segment $\,b+c =1,$
for which Craven and Csordas \cite{CCs} proved that the function
\begin{equation*}
{}_1F_2 \left[\begin{array}{c}
1/2\\ 1/3, \,2/3\end{array}\biggr| \frac{4z}{27} \right] =\sum_{k=0}^\infty \frac{(2k)!}{\,k! (3k)!} z^k
\end{equation*}
belongs to the class $\mathcal{LP}^+$.

\medskip

{\bf 3.} In the last place, let us consider the case $\,a>1/2,\,(b, c)\in \mathcal{Z}_a\,$ for which
$\Phi(z)$ defined by \eqref{HG1} belongs to $\mathcal{LP}$. Since $\,\mathcal{Z}_a = \big(a, \,a+1\big]\times I_a,$
if we apply the transference principle, for fixed $a, c$, then we find that $\,\Phi\in\mathcal{LP}\,$ for
$\,(b, c)\in \big(0, \,a+1\big]\times I_a.$ We now fix $a, b$ and apply the transference principle
to the $c$-parameter range $I_a$ as follows:

\begin{itemize}
\item Let us assume that $\,1/2<a<1\,$ so that $\,I_a =\big[1-a,\,2a-1/2\big].$
Since $\,2a-1/2>1\,$ only when $\,a>3/4,$
we may not extend $I_a$ further in the case $\,1/2<a\le 3/4.$ If $\,a>3/4,$
we may shift down $I_a$ by one unit.
Since $\,2a-3/2\ge 1-a\,$ only when $\,a\ge 5/6,$ this extension of $I_a$
amounts to the union of two intervals
$$\big(0,\,2a-3/2\big]\cup\big[1-a,\,2a-1/2]$$
when $\,3/4<a<5/6\,$ and the interval $\,\big(0, \,2a-1/2\big]\,$ when $\,5/6\le a<1.$

\item In the case $\,a\ge 1,$ we have $\,I_a = \big(0,\,a/2+1\big]\,$
so that shifting down by one unit is meaningless.
If we assume, however, that $\,a\ge 1+m\,$ with $m$ being a nonnegative integer,
we can extend $I_a$ by applying
the transference principle in a reverse way.  Indeed,
the transference principle implies that $\,\Phi\in\mathcal{LP}\,$ when the function
\begin{equation*}
{}_1F_2 \left[\begin{array}{c}
a-m\\ b-m, \,c-m\end{array}\biggr| - \frac{\,z^2}{4} \right]
\end{equation*}
belongs to $\mathcal{LP},$ that is, when
$$0<b-m\le a-m+1,\quad 0<c-m\le \frac 12(a-m) +1.$$
Shifting down $m$ units further, we find that the condition
$$0<b\le a+1,\quad 0<c\le \frac 12(a+m) +1$$
is sufficient for the membership $\,\Phi\in\mathcal{LP},$ provided $\,a\ge 1+m.$
\end{itemize}

What have been proved may be summarized as follows, where the Gaussian symbol
$[\alpha]$ of $\alpha\in\R\,$ denotes the largest integer not exceeding $\alpha$.

\begin{theorem}\label{test2}
For $\,a> 1/2,$ let $\mathcal{X}_a$ be the set of all ordered-parameter pairs
$\,(b, c)\in\R_+^2\,$ defined by
$\,\mathcal{X}_a = \big(0, \,a+1\big]\times L_a,$ where
$$L_a = \left\{\begin{aligned}
&{\,\,\Big[1-a,\, 2a- \frac 12\Big] \quad\text{for}\quad \frac 12<a\le \frac 34,}\\
&{\,\,\Big(0, \,2a-\frac 32 \Big] \cup\Big[1-a,\, 2a-\frac 12\Big]\quad\text{for}\quad \frac 34< a\le\frac 56,}\\
&{\,\,\Big(0, \,2a-\frac 12\Big] \quad\text{for}\quad \frac 56\le a <1,}\\
&{\,\,\Big(0, \,\frac 12\big(a+[a-1]\big) +1\Big] \quad\text{for}\quad  a\ge 1,}\end{aligned}\right.$$
and $\,\mathcal{X}_a^* =\left\{(b, c): (c, b)\in \mathcal{X}_a\right\}.$
If $\,(b, c)\in \big(\mathcal{X}_a \cup \mathcal{X}_a^*\big),$ then $\,\Phi\in \mathcal{LP}.$
\end{theorem}

In the special case $\,a=1,$ this theorem yields in particular
$$
{}_1F_2 \left[\begin{array}{c}
1\\ (r+1)/2, \,(r+2)/2 \end{array}\biggr|\, - \,\right]\in\mathcal{LP}^+
\quad\text{for all $\,-1<r\le 2,$}
$$
the same result established by P\'olya and Hille \cite{Hi}.

To illustrate how all of the above criteria are combined to specify the
range of parameters for $\mathcal{LP}$, we take $\,a=7/2 = 1/2 +3\,$
in the definition of \eqref{HG1}. For convenience, we shall denote by $\Delta$ the set of all ordered-pairs
$(b, c)$ for which $\,\Phi\in \mathcal{LP}$ (see Figure \ref{Fig2}).

\begin{figure}[!ht]
 \centering
 \includegraphics[width=280pt, height= 280pt]{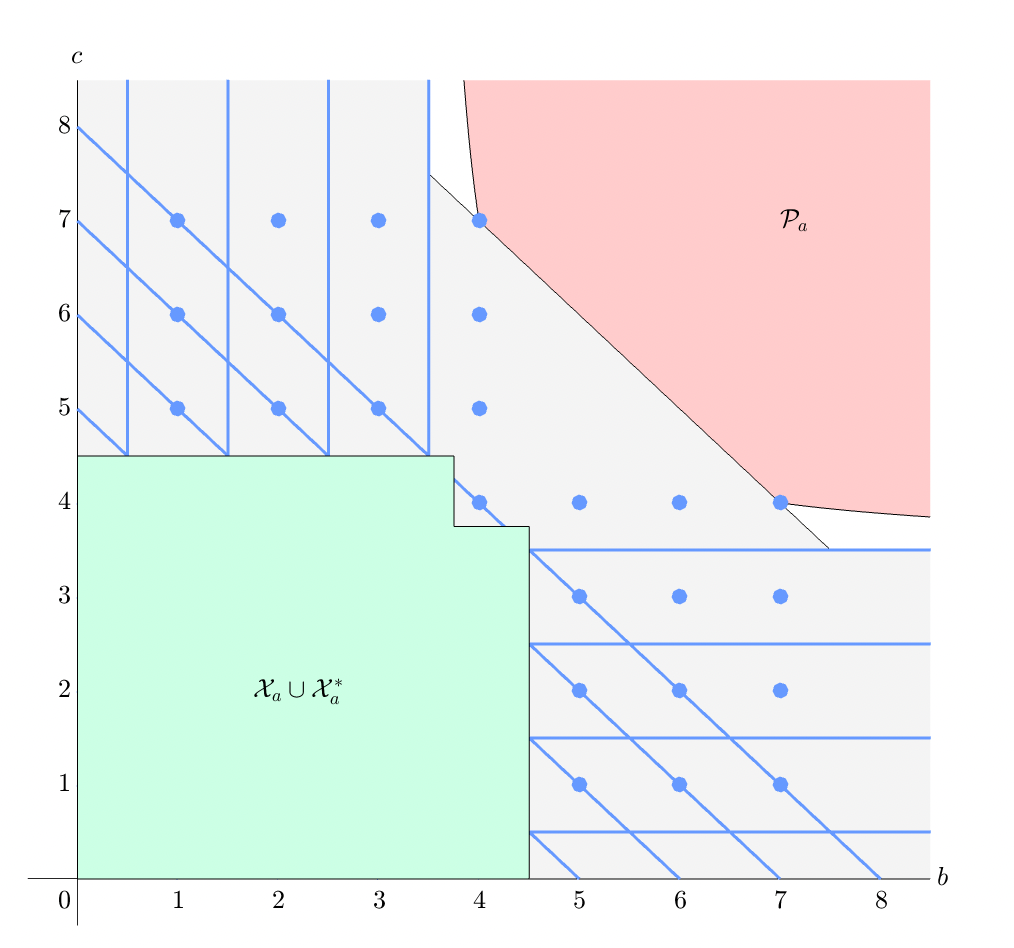}
 \caption{For $\,a=7/2,$ the region $\Delta$ of parameter-pair $(b, c)$ for which
 $\,\Phi\in\mathcal{LP}\,$ consists of line segments, vertical and horizontal rays,
 integer-lattice points and rectangles $\mathcal{X}_a, \mathcal{X}_a^*.$
 If $\,(b, c)\in \mathcal{P}_a\setminus\{(4, 7),\,(7, 4)\},$ then $\,\Phi\notin \mathcal{LP}$ with infinitely
 many complex zeros.}
\label{Fig2}
\end{figure}

\begin{itemize}
\item By Theorem \ref{test1}, if $\,(b, c)\in\mathcal{P}_a\setminus\big\{(4, 7), (7, 4)\big\},$ where
$$\mathcal{P}_a =\big\{ (b, c): b>7/2,\,\,b+c\ge 11, \,\,(2b-7)(2c-7)\ge 7\big\},$$
then $\,\Phi\notin\mathcal{LP}\,$ and $\Phi(z)$ has an infinity of complex zeros.

\item By Proposition \ref{propositionP2}, $\Delta$ contains four vertical rays $\,b=7/2-k,\,c>0,$ and
four horizontal rays $\,b>0,\,c=7/2-k,$ where $\, k=0, 1, 2, 3.$ Similarly, by Proposition \ref{propositionP3},
$\Delta$ contains the lattice points
\begin{align*}
&\left\{(m, n)\in \mathbb{Z}_+^2 : 1\le m\le 4,\, 1\le n\le 7\right\}\\
&\qquad \cup\,\left\{(m, n)\in \mathbb{Z}_+^2 : 1\le m\le 7,\, 1\le n\le 4\right\}.
\end{align*}

\item By Theorem \ref{propositionP4}, $\Delta$ contains eight line segments defined by
$$b+c =k,\,b>0,\,c>0,\quad\text{where $\, k=1, \cdots, 8.$}$$

\item By Theorem \ref{test2}, $\Delta$ contains the rectangles
$$\mathcal{X}_a = \big(0, \, 9/2\big] \times \big(0, \, 15/4\big],\quad
\mathcal{X}_a^* = \big(0, \, 15/4\big]\times \big(0, \, 9/2 \big].$$
\end{itemize}

\bigskip

\noindent
{\bf Acknowledgements.}
This research by Yong-Kum Cho was supported by
the National Research Foundation of Korea funded by
the Ministry of Science and ICT of Korea
(2021R1A2C1007437).


\begin{thebibliography}{12}

\bibitem[1]{BK} \'A. Baricz and S. Koumandos,
\emph{Tur\'an type inequalities for some Lommel functions of the first kind},
Proc. Edinburgh Math. Soc. 59, 569--579 (2016)

\bibitem[2]{CC1} Y.-K. Cho and S.-Y. Chung,
\emph{Collective interlacing and ranges of the positive zeros of Bessel functions},
J. Math. Anal. Appl. 500, 125166 (2021)

\bibitem[3]{CC2} Y.-K. Cho and S.-Y. Chung,
\emph{The Newton polyhedron and positivity of ${}_2F_3$ hypergeometric functions},
Constr. Approx. 54, 353--389 (2021)

\bibitem[4]{CCY} Y.-K. Cho, S.-Y. Chung and H. Yun,
\emph{Rational extension of the Newton diagram for the positivity of ${}_1F_2$ hypergeometric functions
and Askey-Szeg\"o problem},
Constr. Approx. 51, 49--72 (2020)

\bibitem[5]{CCs} T. Craven and G. Csordas,
\emph{The Fox-Wright functions and Laguerre multiplier sequences},
J. Math. Anal. Appl. 314, 109--125 (2006)

\bibitem[6]{DR} D. K. Dimitrov and P. K. Rusev,
\emph{Zeros of entire Fourier transforms},
East J. Approx. 17, 1--108 (2011)


\bibitem[7]{E} A. Erd\'elyi, W. Magnus, F. Oberhettinger and F. G. Tricomi,
\emph{Tables of Integral Transforms, Vol. II}, McGraw-Hill (1954)



\bibitem[8]{Ga} G. Gasper,
\emph{Positive integrals of Bessel functions},
SIAM J. Math. Anal. 6, 868--881 (1975)


\bibitem[9]{GM} G. P. Gupta and M. E. Muldoon,
\emph{Riccati equations and convolution formulae for functions of Rayleigh type},
J. Phys. A 33, 1363--1368 (2000)

\bibitem[10]{H1} J. R. Higgins,
\emph{An interpolation series associated with the Bessel-Hankel transform},
J. London Math. Soc. 5, 707--714 (1972)

\bibitem[11]{H2} J. R. Higgins,
\emph{Five short stories about the cardinal series},
Bull. Amer. Math. Soc. 12, 45--88 (1985)

\bibitem[12]{Hi} E. Hille,
\emph{Note on some hypergeometric series of higher order},
J. London Math. Soc. 4, 50--54 (1929)


\bibitem[13]{IM} M. E. H. Ismail and M. E. Muldoon,
\emph{On the variation with respect to a parameter of zeros of Bessel and q-Bessel functions},
J. Math. Anal. Appl. 135, 187--207 (1988)

\bibitem[14]{J} A. J. Jerri,
\emph{The Shannon sampling theorem-Its various extensions and applications: A tutorial review},
Proc. IEEE 65, 1565--1595 (1977)

\bibitem[15]{LM} A. Laforgia and M. Muldoon,
\emph{Some consequences of the Sturm comparison theorem},
Amer. Math. Monthly 93, 89--94 (1986)
\bibitem[16]{Lo} L. Lorch,
\emph{Some inequalities for the first positive zeros of Bessel functions},
SIAM J. Math. Anal. 24, 814--823 (1993)

\bibitem[17]{Le} B. Y. Levin,
\emph{Lectures on Entire Functions},
Translations of Mathematical Monographs 150, Amer. Math. Soc. (1996)

\bibitem[18]{Lu} Y. L. Luke,
\emph{The Special Functions and Their Approximations, Vol. I.},
Academic Press, New York (1969)
\bibitem[19]{Mc} J. McMahon,
\emph{On the roots of the Bessel and certain related functions,}
Annals of Math. 9, 23--30 (1895)


\bibitem[20]{Ma1} E. Makai,
\emph{On a monotonic property of certain Sturm-Liouville functions},
Acta Math. Acad. Sci. Hungar. 3, 165--172 (1952)


\bibitem[21]{Ma2} E. Makai,
\emph{On zeros of Bessel functions},
Univ. Beograd Publ. Elektrotehn. Fak. Ser. Mat. Fiz., no. 602-633, 109--110 (1978)

\bibitem[22]{Na} H. Nagaoka,
\emph{Diffraction phenomena produced by an aperture on a curved surface},
J. Coll. of Sci. Imp. Univ. Japan IV, 301--322 (1891)
\bibitem[23]{dlmf} F. W. J. Olver, A. B. Olde Daalhuis, D. W. Lozier, B. I. Schneider, R. F. Boisvert,
C. W. Clark, B. R. Miller, B. V. Saunders, H. S. Cohl, and M. A. McClain, eds.,
\emph{NIST Digital Library of Mathematical Functions. https://dlmf.nist.gov/, Release 1.2.0 of 2024-03-15.}

\bibitem[24]{P} G. P\'olya,
\emph{\"Uber die Nullstellen gewisser ganzer Funktionen},
Math. Z. 2, 352--383 (1918)

\bibitem[25]{Sn2} I. N. Sneddon,
\emph{On some infinite series involving the zeros of Bessel functions of the first kind},
Proc. Glasgow Math. Assoc. 4, 144--156 (1960)

\bibitem[26]{So} A. D. Sokal,
\emph{When does a hypergeometric function ${}_pF_q$ belong to the Laguerre-P\'olya class
$LP^+$}?, J. Math. Anal. Appl., 126432 (2022)

\bibitem[27]{S} C. Sturm,
\emph{M\'emoire sur les \'equations diff\'erentielles du second order},
J. Math. Pures Appl. 1, 106--186 (1834)

\bibitem[28]{Wa} G. N. Watson,
\emph{A Treatise on the Theory of Bessel Functions},
Cambridge University Press, London (1922)




\end{thebibliography}
\end{document}